\documentclass[11pt,reqno]{amsart}
\usepackage{ytableau}
\usepackage{amsthm}
\usepackage{amssymb}
\usepackage{latexsym}
\usepackage{multicol}
\usepackage{etoolbox} 
\usepackage{verbatim,enumerate}
\usepackage{accents}
\usepackage[utf8]{inputenc}
\usepackage{hyperref}
\usepackage{amsmath, amscd}
\usepackage{soul}
\usepackage{tikz} 
\usepackage{mathtools}
\usepackage{cleveref}
\usepackage{dirtytalk}
\usepackage{stmaryrd}
\usetikzlibrary{matrix,arrows,decorations.pathmorphing,arrows.meta,calc}
\usepackage{enumitem}
\newlength\cellsize 

\savebox2{%
\begin{picture}(10,10)
\put(0,0){\line(1,0){10}}
\put(0,0){\line(0,1){10}}
\put(10,0){\line(0,1){10}}
\put(0,10){\line(1,0){10}}
\end{picture}}

\newcommand\cellify[1]{\def\thearg{#1}\def\nothing{}%
\ifx\thearg\nothing\vrule width0pt height\cellsize depth0pt%
  \else\hbox to 0pt{\usebox2\hss}\fi%
  \vbox to 10\unitlength{\vss\hbox to 10\unitlength{\hss$#1$\hss}\vss}}

\newlength{\bibitemsep}
\newlength{\bibparskip}
\let\oldthebibliography\thebibliography
\renewcommand\thebibliography[1]{%
  \oldthebibliography{#1}%
  \setlength{\parskip}{\bibitemsep}%
  \setlength{\itemsep}{\bibparskip}%
}

\newcommand\tableau[1]{\vtop{\let\\=\cr
\setlength\baselineskip{-10000pt}
\setlength\lineskiplimit{10000pt}
\setlength\lineskip{0pt}
\halign{&\cellify{##}\cr#1\crcr}}}

\savebox7{% CIRCLE
\begin{picture}(10,10)
  \put(5,5){\circle{9.4}}
\end{picture}}

\newcommand{\cirfy}[1]{\def\thearg{#1}\def\nothing{}%
\ifx\thearg\nothing\vrule width0pt height\cellsize depth0pt%
  \else\hbox to 0pt{\usebox7\hss}\fi%
  \vbox to 10\unitlength{\vss\hbox to 10\unitlength{\hss$#1$\hss}\vss}}

\newcommand\cirtab[1]{\vtop{\let\\=\cr
\setlength\baselineskip{-10000pt}
\setlength\lineskiplimit{10000pt}
\setlength\lineskip{0pt}
\halign{&\cirfy{##}\cr#1\crcr}}}

\theoremstyle{plain}
\newtheorem*{prop}{Proposition}
\newtheorem{thm}{Theorem}
\newtheorem*{lem}{Lemma}
\newtheorem*{cor}{Corollary}%[section]

\theoremstyle{definition}
\newtheorem*{example}{Example}
\newtheorem*{defn}{Definition}

\newtheorem*{rem}{Remark}

\theoremstyle{remark}

\newcommand{\lie}[1]{\mathfrak{#1}}

\newcommand\bc{\mathbb C}
\newcommand\bn{\mathbb N}
\newcommand\bz{\mathbb Z}

\newcommand{\GL}{\operatorname{GL}}

\advance\textwidth by 1.2in  \advance\oddsidemargin by -.6in \advance\evensidemargin by -.6in

\newcounter{cnt}
 \makeatletter
\def\mydggeometry{\makeatletter\dg@YGRID=1\dg@XGRID=20\unitlength=0.003pt\makeatother}
\makeatother \theoremstyle{remark}

\numberwithin{equation}{section}
\makeatletter
\def\section{\def\@secnumfont{\mdseries}\@startsection{section}{1}%
  \z@{.7\linespacing\@plus\linespacing}{.5\linespacing}%
  {\normalfont\scshape\centering}}
\def\subsection{\def\@secnumfont{\bfseries}\@startsection{subsection}{2}%
  {\parindent}{.5\linespacing\@plus.7\linespacing}{-.5em}%
  {\normalfont\bfseries}}
\def\subsubsection{\def\@subsecnumfont{\bfseries}\@startsection{subsubsection}{3}%
  {\parindent}{.5\linespacing\@plus.7\linespacing}{-.5em}%
  {\normalfont\bfseries}}
  \makeatother

\title{Root generated subalgebras of symmetrizable Kac-Moody algebras}
\author{Irfan Habib}
\address{Indian Institute of Science, Bangalore 560012}
\email{irfanhabib@iisc.ac.in}
\author{Deniz Kus}
\address{University of Bochum, Faculty of Mathematics, Universit{\"a}tsstr. 150, Bochum 44801, 
Germany}
\email{deniz.kus@rub.de}
\thanks{This research was supported through the program \say{Oberwolfach Research Fellows} by the Mathematisches Forschungsinstitut Oberwolfach in 2023. The second and third author thank MFO for this opportunity and the superb working conditions.}
\author{R. Venkatesh}
\address{Indian Institute of Science, Bangalore 560012}
\email{rvenkat@iisc.ac.in}
\thanks{D.K. was partially funded by the Deutsche Forschungsgemeinschaft (DFG, German Research Foundation)– grant 446246717.}
\thanks{I.H. was partially supported by DST/INSPIRE/03/2019/000172.}

\begin{document}
\begin{abstract}
The derived algebra of a symmetrizible Kac-Moody algebra $\lie g$ is generated (as a Lie algebra) by its root spaces corresponding to real roots. In this paper, we address the natural reverse question: given any subset of real root vectors, is the Lie subalgebra of $\lie g$ generated by these again the derived algebra of a Kac-Moody algebra? We call such Lie subalgebras root generated, give an affirmative answer to the above question and show that there is a one-to-one correspondence between them, real closed subroot systems and $\pi$-systems contained in the positive system of $\lie g$. Finally, we apply these identifications to all untwised affine types in order to classify symmetric regular subalgebras first introduced by Dynkin in the finite-dimensional setting. We show that any root generated subalgebra associated to a maximal real closed subroot system can be embedded into a unique maximal symmetric regular subalgebra.
\end{abstract}
\maketitle
\section{Introduction}

In one of his influential papers, Dynkin classified semi-simple subalgebras of finite dimensional simple Lie algebras (see \cite{dynkin1952semisimple}). An important key ingredient to achieve this goal was the study and classification of so-called symmetric regular subalgebras; these are by definition invariant under the action of a fixed Cartan subalgebra and the corresponding set of roots is symmetric. Furthermore, for a given finite-dimensional semi-simple Lie algebra $\mathring{\lie g}$ with set of roots $\mathring{\Delta}$, Dynkin obtained bijective correspondences between the following sets: 
\begin{enumerate}
 \item  the set of symmetric regular subalgebras of $\mathring{\lie g}$,
     \item the set of (linearly independent) $ \pi-\text{systems of }\mathring{\Delta}\
             \text{contained in  } \mathring{\Delta}^+$,
     \item the set of $ \text{closed subroot systems of } \mathring{\Delta}$.
 \end{enumerate}
 Using this correspondence, Dynkin reduced the problem of classifying symmetric regular subalgebras to the combinatorial problem of classifying closed subroot systems.  Moreover, any symmetric regular subalgebra corresponds to a Cartan-invariant semi-simple subalgebra of $\mathring{\lie g}$.\par The analogue picture for an arbitrary Kac-Moody algebra $\lie g$ fails and first counterexamples can be constructed even for low rank untwisted affine Lie algebras. Although the structure of symmetric regular subalgebras is not completely determined by the combinatorics of their root systems, the first step of understanding them is to investigate the combinatorial picture. This was achieved in a series of papers by various authors, also due to its further applications to the theory of abelian/ad-nilpotent ideals of Borel subalgebras \cite{RS04} or reflection subgroups of finite and affine Weyl groups \cite{DL11a}, to name a few. In \cite{roy2019maximal} the authors classified (maximal) real closed subroot systems for affine Lie algebras which was later generalized in \cite{KV21a} to affine reflection systems by developing a similar theory as Borel and de Siebenthal in \cite{BdS}. \par
 The motivation of this paper is twofold. On the one hand we aim to attack the algebraic picture and classify symmetric regular subalgebras (for first attempts in the affine case see \cite{feingold2004subalgebras}) and on the other hand we want to establish analogue bijective correspondences in the general Kac-Moody setting. In contrast to the finite case, symmetric regular subalgebras of $\lie g$ do not necessarily correspond to Kac-Moody subalgebras which makes a classification quite challenging. Pursuing both goals has led to the study of \textit{root generated} subalgebras of Kac-Moody algebras (these are subalgebras generated by real root vectors) and the natural question whether a root generated subalgebra is again of Kac-Moody type (see Theorem~\ref{KMtypereg} for an affirmative answer). They form an important subclass of symmetric regular subalgebras and we show in Theorem \ref{mainresbij} that they give rise to the analogue bijective correspondences between the following sets:
 \begin{enumerate}
 \item  the set of root generated subalgebras of $\lie g$
     \item the set of $ \pi-\text{systems of }\Delta\
             \text{contained in  }\Delta^+$
     \item the set of $ \text{real closed subroot systems of } \Delta$
 \end{enumerate}
where $\Delta$ is the root system of $\lie g$ with respect to a fixed Cartan subalgebra. One of the crucial facts used to prove these bijective correspondences is Proposition~\ref{corgen34} where we show that the real roots of the root generated subalgebra $\lie g(\Psi)=\langle \lie g_{\alpha}, \lie g_{-\alpha} : \alpha \in \Psi \rangle$ coincides with $\Psi$ for any real closed subroot system $\Psi$; a priori $\lie g(\Psi)$ could have more real roots. This generalizes the affine Kac-Moody case \cite[Corollary 11.1.5]{roy2019maximal} and the rank two Kac-Moody case when $\Psi$ is generated by a subset of simple roots (see \cite[Proposition 4.1]{carbone2021commutator} for a precise statement).\par 

For the untwisted affine Lie algebras we go a step further and give explicit descriptions of all symmetric regular subalgebras and describe the subclass of maximal ones (see Theorem~\ref{mainthmaffine1} and Theorem~\ref{maintheoremaffinemaximalones}). 

To be more precise, we prove that the maximal symmetric regular algebras are all of the form
$$\lie g(\Psi)\oplus \mathbb{C} d,\ \ \ \left(\text{resp.} \ \ \lie g(\Psi')\oplus \bigoplus\limits_{r\in \mathbb{Z}\backslash \{0\}} \lie g_{r\delta}\oplus \mathbb{C}d\right)$$
for some maximal real closed subroot system $\Psi$ (resp. $\Psi'$) with full gradient (resp. proper gradient).
As a consequence, we see that any root generated subalgebra $\lie g(\Psi)$ for a maximal real closed subroot system $\Psi$ can be embedded into a unique maximal symmetric regular subalgebra. This observation suggests a one-to-one correspondence between them in the general Kac-Moody setting. 

At the end, we make some remarks on the non-symmetric case and prove a decomposition of regular subalgebras of $\lie g$ into a semi--direct product of its symmetric and special part respectively under some natural hypothesis (see Propistion~\ref{dynanalogue}); this generalizes the decomposition that is available in the finite case. A similar type of decomposition can be found in \cite{Caprace}.

\medskip
The paper is organized as follows: we recall the definitions, set up notations and prove basic results in Section~\ref{section2}. Root generated subalgebras and $\pi-$systems are introduced in Section~\ref{section3} and the bijective correspondences are obtained in Section~\ref{section4}. In Section~\ref{section5} we focus on the untwisted affine case and classify all (maximal) symmetric regular subalgebras. We end with some remarks on the non-symmetric case and the discussion of some open problems.  

\section{Preliminaries and root generated Lie algebras}\label{section2}
\subsection{}Throughout this paper we denote by $\mathbb{C}$ the field of complex numbers and by $\mathbb{Z}$ (resp. $\mathbb{Z}_{+}$, $\mathbb{N}$) the subset of integers (resp. non-negative, positive integers). For an index set $I$ let $\bz^I$ (resp. $\bz_+^I$, $\bz_-^I$) be the set of all $I$-tuples of integers (resp. non-negative, non-positive integers) with finitely many non-zero entries. For an indeterminate $t$ we let $\mathbb C[t]$ (resp. $\mathbb C[t^{\pm}]$) be the ring of polynomials (resp. Laurent polynomials) in $t$. By convention, the empty sum is defined to be equal to 0 and for a collection of subsets $\{C_k\}$ of a set $C$ we have $\bigcap_{k\in\emptyset}C_k=C$.
\subsection{} Let $\lie g$ be a symmetrizable Kac-Moody algebra with Cartan subalgebra $\mathfrak{h}$ and bilinear form $(\cdot,\cdot)$. The root system of $\lie g$ with respect to $\lie h$ is denoted by $\Delta$ and the set of real and imaginary roots respectively are denoted by $\Delta^{\mathrm{re}}$ and $\Delta^{\mathrm{im}}$ respectively. We have 
\begin{equation}\label{grklchar}\Delta^{\mathrm{re}}=\{\alpha\in \Delta: (\alpha,\alpha)>0\},\ \ \Delta^{\mathrm{im}}=\{\alpha\in \Delta: (\alpha,\alpha)\leq 0\}.\end{equation} 
We choose a Borel subalgebra $\lie b$ in $\lie g$ and let $\Delta^+$ be the corresponding set of positive roots. Let $W$ be the Weyl group of $\lie g$ which is a Coxeter group generated by reflections $s_{\alpha},\ \alpha\in \Delta^{\mathrm{re}}$ and set $W_{M}=\langle s_{\alpha}: \alpha\in M\rangle$ for any subset $M\subseteq \Delta^{\mathrm{re}}$. We have the following root space decomposition 
$$\lie{g}=\lie h\oplus \bigoplus_{\alpha\in \Delta} \lie{g}_\alpha,\ \ \lie{g}_\alpha=\left\{x\in\lie g: [h,x]=\alpha(h)x\  \forall h\in \lie h\right\}$$
and note that $\dim \lie g_\alpha=1$ for all real roots $\alpha \in \Delta^{\mathrm{re}}$. We fix in the rest of the paper generators $\lie g_\alpha=\mathbb{C}x_\alpha$ for $\alpha \in \Delta^{\mathrm{re}}$ and denote by $\mathfrak{sl}(\alpha)\cong \mathfrak{sl}_2(\mathbb{C})$ the subalgebra of $\lie g$ generated by the triple $\langle x_\alpha, x_{-\alpha},\alpha^\vee \rangle$, where $\alpha^\vee=[x_\alpha, x_{-\alpha}]\in \lie h$. The aim of this article is to study symmetric regular subalgebras of Kac-Moody algebras introduced first by Dynkin in the context of finite-dimensional semi-simple Lie algebras \cite{dynkin1952semisimple}.
\begin{defn} A $\lie h$-invariant Lie subalgebra of $\lie{g}$ is called a \textit{regular subalgebra}. Given a regular subalgebra $\lie{s}$, we define
$$\Delta(\lie{s}):=\{\alpha\in\Delta:(\lie{g}_{\alpha}\cap \lie s)\neq 0\},\ \ \lie{s}_{\alpha}:=\lie{g}_\alpha\cap \lie{s},\ \alpha\in\Delta(\lie s)$$ to be the set of roots of $\lie s$ with respect to $\lie h$.  We call $\lie s$ \textit{symmetric} if $\Delta(\lie{s})=-\Delta(\lie{s})$.
\end{defn}
Note that a regular subalgebra admits the following decomposition
$$\lie{s}=(\lie h\cap \lie s)\oplus \bigoplus_{\alpha\in \Delta(\lie s)}\lie{s}_\alpha.$$ 
In a first step, it is natural to attack the problem of understanding the structure of regular subalgebras in a purely combinatorial way. The set of roots $\Delta(\lie s)$ will contain subsets satisfying certain combinatorial properties (see the introduction of the article) which have been tried to classify in the literature by various authors. This leads to the following definitions. 

\begin{defn}
A non-empty subset $\Psi\subseteq \Delta$ is called 
\begin{itemize}
\item \textit{ symmetric} if $\Psi=-\Psi$.

    \vspace{0,2cm}
    
    \item \textit{ a subroot system} if $s_{\alpha}(\beta)\in \Psi$ for all $\alpha\in \Psi\cap \Delta^{\mathrm{re}}$ and $\beta\in \Psi$.
    
    \vspace{0,2cm}
    
    \item \textit{ closed} if for all $\alpha,\beta\in \Psi$ with $\alpha+\beta\in \Delta$ we have $\alpha+\beta\in \Psi$.

       \vspace{0,2cm}
       
    \item \textit{ real closed} if $\Psi\subseteq \Delta^{\mathrm{re}}$ and the condition
    $$\alpha+\beta\in \Delta^{\mathrm{re}},\ \alpha,\beta\in \Psi \implies \alpha+\beta\in \Psi $$ holds, i.e. the sum of two roots in $\Psi$ is either not a root, imaginary or contained in $\Psi$.
\end{itemize}

\end{defn}
The classification of such subsets is quite challenging and wide open for Kac-Moody algebras. In particular cases, for example when $\Delta$ is a finite root system, $\Psi$ is a maximal real closed subroot system or a symmetric real closed subset of the roots of an (extended) affine Lie algebra, classifications are given in the literature (see for example \cite{BHV23a,DdeG21,KV21a,roy2019maximal} and references therein). However, the predescribed subsets do not determine the structure of regular subalgebras completely. In this article we focus on understanding the structure of so-called root generated subalgebras (see \eqref{subal1} for the definition) and show that their structure is governed by real closed subroot systems (see Theorem~\ref{mainresbij}). Moreover, we will answer the fundamental question when these algebras are again of Kac-Moody type in the same fashion as closed subroot systems of finite root systems are again finite root systems of finite-dimensional semi-simple Lie algebras \cite{dynkin1952semisimple}.

\subsection{} \textit{In the rest of the article we let without further comment $\mathfrak{s}$ to be a symmetric regular subalgebra of a Kac-Moody algebra $\lie g$.} 
We set 
$$\Delta(\lie{s})^{\mathrm{re}}=\Delta^{\mathrm{re}}\cap \Delta(\lie{s}),\ \ \Delta(\lie{s})^{\mathrm{im}}=\Delta^{\mathrm{im}}\cap \Delta(\lie{s}),$$ 
$$\Delta(\lie{s})^{\mathrm{re},+}=\Delta(\lie{s})^{\mathrm{re}}\cap \Delta^+,\ \ \Delta(\lie{s})^{\mathrm{im},+}=\Delta(\lie{s})^{\mathrm{im}}\cap \Delta^+$$
and show in the rest of this section that $\Delta(\lie s)^{\mathrm{re}}$ is a real closed subroot system (see Proposition~\ref{keyprop}). In general, $\Delta(\lie{s})$ is not necessarily closed as the following example shows.
\begin{example}  Let $\Delta$ be of type $A_{5}^{(1)}$ and $\lie{s}=\lie{g}_1\oplus \lie{g}_2$ where $\lie{g}_1$ and $\lie{g}_2$ respectively is the subalgebra generated by the root spaces corresponding to the roots in $S_1$ and $S_2$ respectively where 
$$S_1=\{\pm\alpha+2\bz\delta:\alpha=\alpha_1,\alpha_2,\alpha_1+\alpha_2\},\ \ S_2=\{\pm\alpha+3\bz\delta:\alpha=\alpha_4,\alpha_5,\alpha_4+\alpha_5\}.$$ Then $2\delta,\alpha_4\in \Delta(\lie{s})$ but $\alpha_4+2\delta\in \Delta\backslash \Delta(\lie{s}).$
\end{example}
First we need some elementary results to prepare. The following lemma is standard and is proven for the readers convenience (see also \cite[Lemma 3.7]{marquis2021structure} for the case $\lie s=\lie g$).
\begin{lem}\label{sumroot}
  Let $\alpha, \beta\in \Delta(\lie s)$. 
  \begin{enumerate}
      \item If  $\alpha\in \Delta(\lie s)^{\mathrm{re}}$, then we have $[\lie{s}_\alpha,\lie{s}_\beta]\neq 0$ if and only if $\alpha+\beta\in \Delta(\lie{s}).$
              \item If $\alpha\neq \beta$ and $(\alpha, \beta)<0$, then we have $\alpha+\beta\in \Delta(\lie{s}).$
                            \item If $\alpha\in \Delta(\lie s)^{\mathrm{im}}$ with $\alpha\neq \beta$ and $(\alpha, \beta)<0$, then $\beta+k\alpha \in \Delta(\lie s)$ for all $k\in \mathbb{N}.$

              \item If $\alpha\in \Delta(\lie s)^{\mathrm{im}}$ is a non-isotropic root with $\mathrm{dim}(\lie s_\alpha)>1$, then  $k\alpha\in \Delta(\lie s)^{\mathrm{im}}$ for all $k\in \mathbb{N}.$

  \end{enumerate}
    \end{lem}
    \begin{proof}
        We prove part $(1)$ first.
        Suppose $[\lie{s}_\alpha,\lie{s}_\beta]\neq 0,$ then it is immediate that $\alpha+\beta\in \Delta(\lie s)$.
        Conversely, assume that $\alpha+\beta$ is a root of $\lie s$ and consider the direct sum
        $$M^{\lie s}(\alpha,\beta):=\bigoplus_{k\in \bz}\lie s_{\beta+k\alpha}.$$ Since $\mathfrak{sl}(\alpha)\subseteq \lie s$ ($\lie s$ is symmetric) we have that $M^{\lie s}(\alpha,\beta)$ is a  $\mathfrak{sl}(\alpha)$--module and finite-dimensional by \cite[Proposition 5.1(c)]{kac1990infinite}. If $[\lie{s}_{\alpha},\lie{s}_{\beta}]=0$, then each non-zero vector in $\lie{s}_{\beta}$ is a maximal vector for the $\mathfrak{sl}(\alpha)$ action and hence $(\beta,\alpha^\vee)\geq 0$; in particular $\lambda:=(\beta,\alpha^\vee)+2>0$. Now 
        since $\alpha+\beta$ is a root, we can choose a non-zero vector $v\in\lie{s}_{\alpha+\beta}$ which is an eigenvector for the action of $\alpha^\vee$ with eigenvalue $\lambda$. Since $\lambda>0,$  we have from $\mathfrak{sl}_2$-theory that $0\neq x_{-\alpha} .v=[x_{-\alpha} ,v]\in \lie{s}_{\beta}$ and hence $[x_\alpha ,[x_{-\alpha} ,v]]\in [\lie{s}_{\alpha},\lie{s}_{\beta}]$ is a non-zero scalar multiple of $v$ (again from $\mathfrak{sl}_2$-theory). This is a contradiction and the proof is completed.
        For part (2), say $\alpha, \beta\in \Delta(\lie s)$ satisfying $\alpha\neq \beta$ and $(\alpha, \beta)<0$. We have $[\lie s_\alpha, \lie s_\beta]\neq 0$ 
        from \cite[Theorem A]{marquis2021structure}, and hence $\alpha+\beta\in \Delta(\lie{s}).$     Part $(3)$ follows from the repeated use of part $(2)$. For part $(4)$ we use again \cite[Theorem A]{marquis2021structure} to get
        $2\alpha \in \Delta(\lie s)^{\mathrm{im}}$ as $(\alpha, \alpha)<0$ and there exists $x, y\in \lie s_\alpha$ such that $\mathbb{C}x\neq \mathbb{C}y.$ Now repeated use of
        \cite[Theorem A]{marquis2021structure} proves that $k\alpha\in \Delta(\lie s)^{\mathrm{im}}$ for all $k\in \mathbb{N}$ as  $(\alpha, k\alpha)<0.$
       This completes the proof.
    \end{proof}

    The assumption of the lemma that $\alpha$ is a real root is necessary. For example, for an affine Kac-Moody algebra  we could take $\alpha,\beta\in \Delta^{\mathrm{im}}$ such that $\beta\neq -\alpha$. Then we have  $\alpha+\beta \in \Delta$, but $[\lie g_\alpha, \, \lie g_\beta]=0$. 
  \subsection{}   We need one more straightforward lemma.
    \begin{lem}\label{sumimaginary}
    Let $\alpha, \beta\in \Delta(\lie s)^{\mathrm{re}}$ such that $\beta+k\alpha\in \Delta(\lie s)^{\mathrm{im}}$ for some $k\in \pm \mathbb{N}.$ Then we have $(\mathrm{ad}\, x)^{|k|+1}(y)\neq 0$ for any non-zero $x \in \lie s_{\pm \alpha}$ and $y \in \lie s_\beta.$ 
    \end{lem}
    
    \begin{proof} Let $k\in \mathbb{N}.$
        Since $\beta+k\alpha$ is an imaginary root we have $(\beta+k\alpha,\beta+k\alpha)\le 0$ which implies 
        $$-(\beta,\alpha^\vee)\geq k+\frac{(\beta,\beta)}{k(\alpha,\alpha)}\implies -(\beta,\alpha^\vee) \geq k+1$$
where the second inequality is a consequence of $(\beta,\alpha^\vee)\in\bz$ and $\frac{(\beta,\beta)}{(\alpha,\alpha)}>0$ (both roots are real).
 Now the result follows from the representation theory of $\mathfrak{sl}(\alpha).$ The case $k<0$ is done similarly.
\end{proof}
For $\alpha\in \Delta^{\mathrm{re}}$ and $\beta\in \Delta$, we denote by $\mathbb{S}(\alpha,\beta)$ by the $\alpha$-string through $\beta$. Then $\mathbb{S}(\alpha,\beta)$ is a finite, connected subset of $\Delta$ given as follows
$$\mathbb{S}(\alpha,\beta)=\{\beta-p\alpha,\dots,\beta-\alpha,\beta,\beta+\alpha,\dots,\beta+q\alpha\}$$
%,\ \ \mathbb{S}(\alpha,\beta)^+:=\{\beta,\beta+\alpha,\dots,\beta+q\alpha\}.$$ 
where $p$ and $q$ are non-negative integers such that $p-q=(\beta,\alpha^\vee)$ (see \cite{kac1990infinite} for details). Furthermore, set $\mathbb{S}(\alpha,\beta)^+:=\{\beta,\beta+\alpha,\dots,\beta+q\alpha\}$. It is known that $|\mathbb{S}(\beta,\alpha)\cap\Delta^{\mathrm{re}}|\le 4$ (see \cite[Exercise 5.14]{kac1990infinite}) and in fact the following condition was proved by Morita in \cite{morita1988root}:
\begin{equation}\label{thmmorita}
|\mathbb S(\alpha,\beta)\cap \Delta^{\mathrm{re}}|\in\{3,4\}\ \text{for some $(\alpha,\beta)\in \Delta^{\mathrm{re}}\times \Delta$ $\iff \exists i,j\in I$: $a_{ij}=-1$, $a_{ji}<-1$} 
\end{equation}
For a more explicit description where the real roots occur in a root string, the reader is referred to \cite[Proposition 1]{billig1995root}. The following result is important in what follows next.
    \begin{prop}\label{keyprop}
    Let $\lie{s}$ be a regular symmetric subalgebra of $\lie{g}$ and $\alpha,\beta\in \Delta(\lie{s})^{\mathrm{re}}$. 
    \begin{enumerate}
    \item\label{realclosed}  If $\alpha+\beta\in \Delta,$ then we have $\alpha+\beta\in \Delta(\lie{s}).$
        \vspace{0,1cm}
            \item\label{realsubroot} We have $s_\alpha(\beta)\in \Delta(\lie{s}).$
\vspace{0,1cm}
        \item\label{containsrootstring} We have $\mathbb{S}(\alpha,\beta)\subseteq \Delta(\lie{s}).$
      \end{enumerate}
    In particular, $\Delta(\lie{s})^{\mathrm{re}}$ is a real closed subroot system of $\Delta.$
    \end{prop}
    \begin{proof}
    \begin{enumerate}
        \item Since $\alpha,\beta$ are both real roots, we have $\lie{s}_\alpha=\lie{g}_\alpha,\ \lie{s}_\beta=\lie{g}_\beta.$ Since $\alpha+\beta\in \Delta,$ the result follows from Lemma \ref{sumroot}.
        
        \item If $(\beta,\alpha^\vee)=0,$ then $s_\alpha(\beta)=\beta$
        and there is nothing to show. 
        Since $s_{\alpha}(\beta)=s_{-\alpha}(\beta)$ we can assume without loss of generality that $(\beta,\alpha^\vee)< 0.$ 
Let $M^{\lie s}(\alpha,\beta)$ be the $\mathfrak{sl}(\alpha)$ submodule of $\lie{s}$ as in the proof of Lemma~\ref{sumroot}; recall that $M^{\lie s}(\alpha,\beta)$ is finite-dimensional. 
From the representation theory of $\mathfrak{sl}(\alpha)$ we have that $$(\mathrm{ad}\, x_\alpha)^{-(\beta,\alpha^\vee)}(x_{\beta})=[x_\alpha,[x_\alpha,\dots [x_\alpha,[x_\alpha, x_{\beta}]]\cdots ]\neq 0$$
and thus 
 $$\left\{\beta,\beta+\alpha,\dots,\beta-(\beta,\alpha^\vee)\alpha\right\}\subseteq \Delta(\lie{s}).$$ 
   \item We will only show that $\mathbb{S}(\alpha,\beta)^+\subseteq \Delta(\lie{s}).$ The statement is straightforward provided that $\mathbb{S}(\alpha,\beta)^+\subseteq \Delta^{\mathrm{re}}$. So let $k$ be the largest integer such that $\beta+k\alpha$ is an imaginary root.  From Lemma~\ref{sumimaginary} we immeadiately have 
$$\{\beta,\beta+\alpha,\dots,\beta+(k+1)\alpha\}\subseteq \Delta(\lie{s}).$$  Since the remaining roots $\beta+(k+1)\alpha,\beta+(k+2)\alpha,\dots,\beta+q\alpha$ are all real, the fact $\mathbb{S}(\alpha,\beta)^+\subseteq \Delta(\lie{s})$ follows by repeatedly applying Lemma \ref{sumroot}.
\end{enumerate}        
    \end{proof}
    \begin{rem}\label{remrootstring}
    The assumption $\beta\in \Delta(\lie s)^{\mathrm{re}}$ is needed in the statement of Proposition~\ref{keyprop}(3). For example, let $\lie{g}$ be an untwisted affine Lie algebra and $\alpha$ be a root of the underlying finite-dimensional simple Lie algebra $\mathring{\lie g}$ with Cartan subalgebra $\mathring{\lie h}$. Then we have $\mathbb{S}(\alpha,\delta)=\{-\alpha+\delta,\delta,\alpha+\delta\}$ where $\delta$ is the unique indivisible positive imaginary root of $\lie{g}$.  For $h\in\mathring{\lie h}$ 
    such that $\alpha(h)=0$ there is a regular symmetric subalgebra 
    $$\lie{s}=\bc \lie{g}_{\pm\alpha}\oplus \bc\alpha^\vee \oplus \bc h\otimes t^{\pm 1}\oplus \bc c$$
    where $c$ is the canonical generator of the center of $\lie g$. However $\mathbb{S}(\alpha,\delta)\not\subseteq \Delta(\lie{s})=\{\pm\alpha,\pm\delta\}.$  
\end{rem}

\section{Subroot systems, Root generated subalgebras and \texorpdfstring{$\pi$}{pi}-systems}\label{section3}
In this section we recall the notion of a $\pi$-system and address existence and uniqueness questions in our setting. 
\subsection{} Recall that the Weyl group $W$ of a Kac-Moody algebra $\lie g$ is a Coxeter group \cite[Proposition 3.13]{kac1990infinite}. Apart from its representation on $\lie h^*$ there is also the so-called geometric representation $\rho:W\to \GL(V)$ (see \cite{DV82} for an exposition) of $W$. In general the set of real roots of $W$ in its geometric representation is drastically different from the real roots of the Kac-Moody algebra $\lie g$ which is emphasized by the next example.
    
    \begin{example}
        This example is adapted from \cite[Exercise 5.27]{kac1990infinite}. Consider the generalized Cartan matrix $\begin{pmatrix}2 & -4\\-1 &2 \end{pmatrix}$ and let $\lie g$ be its Kac-Moody algebra. The positive real roots of $\lie g$ are given by 
        
        \begin{align}\label{realrootsexample}
           \Delta^+\cap \Delta^{\mathrm{re}}=&\left\{2j\alpha_1+(j+1)\alpha_2,\ (j+1)\alpha_1+\frac{1}{2}j\alpha_2:j\in 2\bz_+\right\}\bigcup\nonumber\\
            &\bigcup\left\{j\alpha_1+\frac{1}{2}(j+1)\alpha_2,\ 2(j+1)\alpha_1+j\alpha_2:j\in 2\bz_++1\right\}.
        \end{align}
        The Weyl group $W$ of $\lie g$ has the Coxeter matrix $\begin{pmatrix}1 & \infty\\\infty &1 \end{pmatrix}$ and the geometric representation  $V$ has a basis $e_1,e_2$ with bilinear form $$(e_i,e_i)=1 \ \ i=1,2;\ \ \ (e_1,e_2)=-1.$$ Since all real roots in $V$ are unit vectors, a real root $\alpha=ae_1+be_2$ must satisfy $(a-b)^2=1$ . Hence except for a very few cases, the roots appearing in equation \eqref{realrootsexample} are not real roots in $V$ if we replace $\alpha_i$ by $e_i$ for $i=1,2.$
    \end{example}
 Nevertheless, the same proof as in \cite{deodhar1989note} shows that for any subroot system $\Psi\subseteq \Delta^\mathrm{re}$ the subgroup $W_\Psi$ of the Weyl group $W$ is again Coxeter group with a canonical set of generators $\Pi(\Psi)$ which is constructed in \cite{deodhar1989note} as follows and can be adopted here. Define  
    $$\Omega=\Delta^+\cap\{w(\alpha): w\in W_{\Psi},\ \alpha\in \Psi\}\subseteq \Delta^+\cap \Delta^{\mathrm{re}}$$ and consider the partial order (which depends on $\Omega$)
    $$\gamma_1\preceq \gamma_2 \iff \gamma_2=a\gamma_1+\sum_{\tau\in \Omega\backslash{\{\gamma_1,\gamma_2\}}}a_{\tau}\tau,\ \ \text{for some } a\in \mathbb{Q}_{>0},\ a_{\tau}\in\mathbb{Q}_{\geq 0}.$$
Then $(W_{\Psi},S_{\Psi})$ is a Coxeter group where $S_{\Psi}:=\{s_{\beta}: \beta\in \Pi(\Psi)\}$ and $\Pi(\Psi) \subseteq \Psi\cap \Delta^+$ denotes the set of minimal elements with respect to the above order, i.e. if $\gamma_1\preceq \gamma_2$ with $\gamma_1\in \Psi\cap \Delta^+$ and $\gamma_2\in \Pi(\Psi)$, then $\gamma_1=\gamma_2$. From the definition it is straightforward check that we have the alternative description \begin{equation}\label{altdes1}\Pi(\Psi)=\left\{\beta\in \Psi\cap \Delta^+: s_\beta((\Psi\cap\Delta^+)\backslash\{\beta\})=(\Psi\cap \Delta^+)\backslash\{\beta\}\right\}.\end{equation}
Moreover, for $\alpha,\beta\in\Pi(\Psi)$ with $\alpha\neq\beta$ we have $(\alpha,\beta)\leq 0$ since $(\alpha,\beta)>0$ would imply 
          $$\alpha=(\alpha,\beta^{\vee})\beta+s_{\beta}(\alpha)\implies \beta\preceq \alpha$$
   which is not possible. Last, we have that $W_{\Pi(\Psi)}(\Pi(\Psi))=\Psi$ which we will use without further comment in the rest of the article. 
   %{\color{blue} Follows from the proof of Deodhar, Step $3$.}
\subsection{} We record the following definition. 
\begin{defn}
        A $\pi$\textit{-system} $\Sigma$ of a subroot system $\Psi$ is a subset of $\Psi\cap \Delta^{\mathrm{re}}$ which satisfies the condition $\alpha-\beta\notin\Psi$ for all $\alpha,\beta\in \Sigma.$
    \end{defn}
    
We emphasize that there seems to be no uniform definition of a $\pi$-system in the literature. In \cite{dynkin1952semisimple} a $\pi$-system is required to be linearly independent and in \cite{feingold2004subalgebras} a $\pi$-system doesn't need to be linearly independent, however should contain only positive roots. The authors of \cite{carbone2021varvec} drop both restrictions. Note that a $\pi$-system of $\Delta$ is by definition a $\pi$-system of any subroot system containing the $\pi$-system. \par

The next lemma shows that the above constructed set of minimal elements $\Pi(\Psi)$ gives under a suitable condition on $\Psi$ a $\pi$-system of $\Delta$.
  
  \begin{lem}\label{baseisapisystem}Let $\Psi\subseteq \Delta^{\mathrm{re}}$ be subroot system. \begin{enumerate}[leftmargin=*]
      \item The above constructed set $\Pi(\Psi)$ is a $\pi$-system of $\Psi$. If $\Psi$ is a real closed subroot system, then $\Pi(\Psi)$ is a $\pi$-system of $\Delta$.
    
      \item Let $\Sigma$ be a $\pi$-system of $\Delta$. Then any element of $W_{\Sigma}(\Sigma)$ can be written as a non-negative or non-positive integral linear combination of elements in $\Sigma$.
  \end{enumerate} 
  \begin{proof} We first prove part (1).  Let $\Psi$ be a subroot system and $\alpha,\beta\in \Pi(\Psi)$ such that $\gamma:=\alpha-\beta\in \Psi.$ Without loss of generality we can assume $\gamma\in \Psi\cap \Delta^{+}$ (otherwise we replace $\gamma$ by $-\gamma$). If $\alpha=2\beta$, then $\alpha\succeq \beta$. Otherwise $\gamma\in\Omega\backslash{\{\alpha,\beta}\}$ and again $\alpha=\beta+\gamma$ gives $\alpha\succeq \beta$. This implies that $\alpha=\beta$ which is a contradiction.\par
  Now assume that $\Psi$ is a real closed subroot system and $\alpha-\beta\in \Delta$ for some $\alpha,\beta\in \Pi(\Psi)$. From the discussion preceding the lemma we have $(\alpha,\beta)\le 0$ and therefore $\alpha-\beta\in \Delta^{\mathrm{re}}$ since
    $$(\alpha-\beta,\alpha-\beta)=(\alpha,\alpha)+(\beta,\beta)-2(\alpha,\beta)>0.$$  Now $\Psi$ is real closed in $\Delta,$ implying that $\alpha-\beta\in \Psi$ and we have a contradiction. \vspace{0,1cm}
    
    For the second part let $\Sigma=\{\beta_i: i\in I\}$ and consider the subalgebra $\lie g_{kr}$ of $\lie g$ generated by $\lie g_{\pm\beta_k}=\bc x_{\pm\beta_k}$ and $\lie g_{\pm\beta_r}=\bc x_{\pm \beta_r}.$ Let $W_{kr}$ be the subgroup of $W_{\Sigma}$ generated by $s_{\beta_k}$ and $s_{\beta_r}$ and $\ell_{\{k,r\}}$ be the length function on $W_{kr}.$ Since $\Sigma$ is a $\pi$-system of $\Delta$, we can define the generalized Cartan matrix $B_{kr}=(b_{ij})$ by 
    $$b_{11}=b_{22}=2,\  b_{21}=(\beta_k,\beta_r^\vee),\  b_{12}=(\beta_r,\beta_k^\vee)$$
    and let $\lie g'(B_{kr})$ be the derived algebra of the Kac-Moody algebra associated to $B_{kr}$. Obviously $\{\beta_k,\beta_r\}$ is a linearly independent $\pi$-system of $\Delta$ and thus we get from \cite[Theorem 2.5 and Proposition 2.6]{carbone2021varvec} an isomorphism 
   $$ \varphi:\lie g'(B_{kr})\to \lie g_{kr}$$ 
    $$e_1\mapsto x_{\beta_k}, \ \ e_2\mapsto x_{\beta_r},\ \ f_1\mapsto x_{-\beta_k},\ \ f_2\mapsto x_{-\beta_r},\ \ h_1\mapsto \beta_k^\vee,\ \ h_2\mapsto \beta_r^\vee$$
  Hence the Weyl group of $\lie g'(B_{kr})$ is isomorphic to $W_{kr}$. In particular, $$\ell_{\{k,r\}}(ws_{\beta_i})> \ell_{\{k,r\}}(w),\ w\in W_{k\ell} \implies w(\beta_i)\in \Delta^+.$$
Now the same proof as in \cite[Theorem 5.4]{humphreys1992reflection} (the proof confirms also the cardinality two case first in the geometric setting) shows that the same statement holds for the length function
$\ell_\Sigma$ on $W_{\Sigma}$ with respect to the generators $\Sigma$, i.e., $$\ell_\Sigma(ws_{\beta_k})> \ell_\Sigma(w),   \ w\in W_{\Sigma}\Longrightarrow w(\beta_k)\in \sum_{i\in I}\mathbb{Z}_+\beta_i.$$ 
Now assume that $\beta\in W_{\Sigma}(\Sigma)$ and write $\beta=w(\beta_k)$ for some $k\in I$ and $w\in W_{\Sigma}$. If $\ell_\Sigma(ws_{\beta_k})> \ell_\Sigma(w)$ we have that $\beta=\sum_i c_i\beta_i$ for some $c_i\in\bz_+$ and otherwise $\ell_\Sigma(ws_{\beta_k})< \ell_\Sigma(w)$ which gives $\beta=-ws_{\beta_k}(\beta_k)=-\sum_i c_i\beta_i$ for some $c_i\in\mathbb{Z}_+$.
      \end{proof}
  \end{lem}
 We end this subsection with an example of a $\pi$-system for a hyperbolic Kac-Moody algebra studied by Feingold and Nicolai in \cite{feingold2004subalgebras}. This example also shows that a $\pi$-system need not to be linearly independent.
 \begin{example}\label{FN2d}
     Let $A$ be the hyperbolic GCM $$A=\begin{pmatrix}
          2 & -1 & 0\\
          -1 & 2 & -2\\
          0 & -2 & 2
      \end{pmatrix}$$ 
with simple roots $\{\alpha_i:1\le i\le 4\}$ of $\lie g(A)$ and set $\Sigma=\{\gamma_i:1\le i\le 4\}$ where $$\gamma_1=\alpha_1+\alpha_2,\ \ \gamma_2=2\alpha_1+2\alpha_2+3\alpha_3,\ \ \gamma_3=2\alpha_2+3\alpha_3,\ \ \gamma_4=4\alpha_2+3\alpha_3.$$
      Then $\Sigma$ is a $\pi$-system of $\Delta$. The only non-trivial case is to show that $\gamma_2-\gamma_1\notin \Delta$; all other differences involve either different signs (e.g. $\gamma_4-\gamma_1$) or are multiples of real roots (e.g. $\gamma_2-\gamma_3$). 
      However, if $\gamma_2-\gamma_1\in\Delta$ we would have (root strings are unbroken)
$$\{\gamma_1,\gamma_1+\alpha_3,\gamma_1+2\alpha_3,\gamma_1+3\alpha_3\}\subseteq \mathbb{S}(\alpha_3,\gamma_1)$$
which is impossible by \eqref{thmmorita} since $\gamma_1,\gamma_1+2\alpha_3,\gamma_1+3\alpha_3\in\Delta^{\mathrm{re}}$ by \eqref{grklchar}.  Note that we have the relation $\gamma_4= 2\gamma_1-\gamma_2+2\gamma_3$.
  \end{example}
 
\subsection{} The following is one of the main results of this section.
\begin{prop}\label{uniqq}
    Let $\Psi\subseteq \Delta^{\mathrm{re}}$ be a subroot system of $\Delta.$ 
    Then there exists at most one $\pi$-system $\Sigma$ of $\Delta$ satisfying
    $$\Sigma\subseteq\Delta^+,\ \ W_\Sigma(\Sigma)=\Psi.$$
    
In particular, if $\Psi$ is a real closed subroot system, then $\Pi(\Psi)$ is the unique $\pi$-system of $\Delta$ with the above properties.  
\end{prop}
\begin{proof} Assume that there exists such a $\pi$-system $\Sigma$ with the above mentioned properties. From Lemma~\ref{baseisapisystem} we have that each root in $\Psi$ can be written as a linear combination of roots in $\Sigma$ such that all the coefficients are either non-negative or non-positive integers. In particular, any element $\beta\in \Pi(\Psi)\subseteq \Psi\cap\Delta^+$ can be written as $\beta=\sum_{i=1}^k \sigma_i$ for some roots $\sigma_i\in \Sigma.$ Therefore we have $\beta\succeq \sigma_i$ for all $1\le i\le k$ which is only possible if $k=1$ and thus $\beta\in \Sigma$. So we have shown $\Pi(\Psi)\subseteq \Sigma$ and the remaining part shows the other inclusion. Let $\sigma\in \Sigma$ and $\gamma\in \Psi\cap\Delta^+$ such that $\gamma\neq \sigma$; so we can write $\gamma=\sum_{i=1}^n a_i\sigma_i$ for some $a_i\in \bz_+$ and $\sigma_i\in \Sigma.$ Without loss of generality assume that $\sigma_1=\sigma.$ Since $\Psi$ is a subroot system we must have $s_\sigma(\gamma)\in \Psi.$ In order to show $\sigma\in \Pi(\Psi)$ we will prove that $s_\sigma(\gamma)\in \Psi\cap \Delta^+$ (compare with \eqref{altdes1}). So assume in the rest of the proof that $s_\sigma(\gamma)\in -\Delta^+$ and again by Lemma~\ref{baseisapisystem} we can write 
\begin{equation}\label{funnyeq1}s_{\sigma}(\gamma)=\left( -a_1-\sum_{i=2}^n a_i(\sigma_i,\sigma_1^\vee)\right)\sigma_1+\sum_{i=2}^n a_i\sigma_i=b_1\tau_1+b_2\tau_2+\dots +b_k\tau_k.\end{equation}
for some non-positive integers $b_i$ and roots $\tau_i\in \Sigma$. Define for convenience $a_i'$ to be the coefficient of $\sigma_i$ in the above expression for $1\le i \le n.$ In particular $a_1'<0$ and $(\gamma,\sigma^{\vee})>0$. If $\sigma_1\notin \{\tau_i:1\le i\le k\},$ then we obtain from \eqref{funnyeq1} an expression of the form 
    $$(-a_1')\sigma_1=a_2'\sigma_2+\dots+a_n'\sigma_n+(-b_1)\tau_1+\dots+(-b_k)\tau_k.$$
If $\sigma_1\in  \{\tau_i:1\le i\le k\}$, say $\sigma_1=\tau_1$, and $b_1>a_1'$ then \eqref{funnyeq1} can be rewritten as
    $$(b_1-a_1')\sigma_1=a_2'\sigma_2+\dots+a_n'\sigma_n+(-b_2)\tau_2+\dots+(-b_k)\tau_k$$
    and if $b_1\leq a_1'$ \eqref{funnyeq1} can be rewritten as 
    $$(a_1'-b_1)\sigma_1+a_2'\sigma_2+\dots+a_n'\sigma_n+(-b_2)\tau_2+\dots+(-b_k)\tau_k=0.$$

In any case, we have one of the following two expressions:
\begin{equation}\label{funnyeq2}
    \sum_{i} a_i\mu_i=0\ \ \text{ for some } a_i\in \bz_+\text{ (not all zero) and } \mu_i\in \Sigma,
\end{equation}
or 
\begin{equation}\label{funnyeq3}
    a\mu'=\sum_{i} a_i\mu_i,\ \ \text{ for some } a>0, a_i\in \bz_+\text{ and distinct } \mu',\mu_i\in \Sigma.
\end{equation}

Since $\Sigma\subseteq \Delta^+,$ expressing each $\mu_i$ in \eqref{funnyeq2} as a non-negative linear combination of simple roots of $\lie g$  we get that $a_i=0$ which is a contradiction. Now from \eqref{funnyeq3} we have 
$$0<a(\mu',\mu')=\sum_i a_i(\mu_i,\mu')\le 0$$
which is again a contradiction, where the first inequality holds since $\mu'\in\Delta^{\mathrm{re}}$ and the second inequality follows from the fact that $\Sigma$ is a $\pi$-system and hence $(\alpha,\beta)\le 0$ for distinct $\alpha, \beta\in \Sigma.$ Consequently we have that $s_\sigma(\gamma)\in \Psi\cap\Delta^+$ and so $\sigma\in \Pi(\Psi).$ 
\end{proof}

\begin{rem} If $\Psi$ is not closed, then there might not exist such a $\pi$-system of $\Delta$ satisfying the properties of Proposition~\ref{uniqq}. For example, let $\Delta$ be the finite root system of type $G_2$ and let $\Psi$ be the subroot system consisting of the short roots of $\Delta.$ Then $\Psi$ is of type $A_2$ and there is no subset (except the single point subsets) of $\Psi$ which is a $\pi$-system of $\Delta$. 
\end{rem}
We discuss in the rest of this subsection also an example of an infinite $\pi$-system of $\Delta$. We first need to record the next lemma which generalizes the rank $2$ case in \cite[Theorem 1.1(i)]{carbone2021commutator}.
\begin{lem}\label{sumnotreal}
    If $a_{ij}\le -2$ for all $i\neq j,$ then $\alpha+\beta\notin \Delta^{\mathrm{re}}$ for any $\alpha,\beta\in \Delta^{\mathrm{re}}.$ 
    In particular, any subset $\Psi\subseteq \Delta^{\mathrm{re}}$ is real closed. 
    \end{lem}
    \begin{proof}
        If possible assume that $\alpha,\beta\in\Delta^{\mathrm{re}}$ are such that $\alpha+\beta\in\Delta^{\mathrm{re}}.$ By \eqref{thmmorita} we have  $|\mathbb{S}(\beta,\alpha)\cap\Delta^{\mathrm{re}}|= 2$ and \cite[Proposition 1]{billig1995root} implies $$\mathbb{S}(\beta,\alpha)=\{\beta,\beta+\alpha\},\ \ \ \mathbb{S}(\alpha,\beta)=\{\alpha,\alpha+\beta\}.$$
    Without loss of generality we can assume that $\alpha,\beta\in\Delta^+.$ By replacing, if necessary, the pair $\{\beta,\alpha\}$ by $\{w(\beta),w(\alpha)\}$ for $w\in W$ we can suppose from the beginning that $\{\beta,\alpha\}$ satisfies
    $$\mathrm{ht}(\alpha+\beta)\le \mathrm{ht}(w(\alpha+\beta)),\ \ \forall w\in \{w\in W: w(\alpha),w(\beta)\in\Delta^+\}.$$ By \cite[Proposition 2(ii)]{billig1995root} it must be of the form $\{\beta,\alpha_i\}$ for some simple root $\alpha_i.$ Since 
    $a_{ij}\leq -2$ for all $i,j\in I$ the root $\beta$ cannot be simple, since otherwise $\beta+2\alpha_i\in \mathbb{S}(\beta,\alpha)$. Now by \cite[Proposition 5]{billig1995root} there exists a unique $j\neq i$ such that $(\beta,\alpha_j^\vee)>0$ and $a_{i,j}=-1$ which is impossible. This completes the proof. 
    \end{proof} 
    \begin{example} 
    Let $A=(a_{ij})$ be a $3\times 3$ GCM given by $$a_{ij}=\begin{cases}
        2 & \ \text{if } i=j\\
        -2 & \ \text{if } i\neq j
        \end{cases}$$
        and let $\{\alpha_1,\alpha_2,\alpha_3\}$ be a simple system for its Kac-Moody algebra $\lie g$.  For $k\in \bz$ let 
      $$\gamma_k=2k(2k+1)\alpha_1+2k(2k-1)\alpha_2+\alpha_3=(s_1s_2)^{k}(\alpha_3)\in \Delta^{+}\cap \Delta^{\mathrm{re}}$$
      and set 
      $$\Sigma:=\{\gamma_k:k\in\bz\},\ \ \Psi:=W_{\Sigma}(\Sigma)$$ 
      It is straightforward to check that $\Psi$ is a subroot system consisting of real roots and hence it is a real closed subroot system by Lemma~\ref{sumnotreal}. We will show in the remaining part that $\Sigma$ is a $\pi$-system of $\Delta$ and hence $\Sigma=\Pi(\Psi)$ (see Proposition~\ref{uniqq}) is neither finite nor linearly independent. It is easy to see that $(\gamma_k,\gamma_\ell^\vee)=2-16(k-\ell)^2<-1$ if $k\neq \ell$. In particular, $\mathbb{S}(\gamma_\ell,\gamma_k)$ contains at least two real roots namely $\gamma_k$ and $s_{\gamma_\ell}(\gamma_k).$ By Equation \ref{thmmorita} we have that $|\mathbb{S}(\gamma_\ell,\gamma_k)\cap\Delta^{\mathrm{re}}|=2.$ Since the leftmost and rightmost roots of a root string are real roots (see \cite[Proposition 1]{billig1995root}) it follows that $\gamma_k-\gamma_\ell\notin \Delta.$ 
      \end{example}
     
\subsection{} Given a subset $S\subseteq \Delta^{\mathrm{re}}$, we define
\begin{equation}\label{subal1}\lie g(S)=\langle \lie g_{\alpha}: \alpha\in \pm S\rangle\end{equation} and call it a \textit{root generated} subalgebra. Since the derived algebra of $\lie g$ is itself root generated, it is natural to ask the converse whether all root generated subalgebras are of Kac-Moody type. First we note that $\lie g(S)$ is in fact a symmetric regular subalgebra. To see this, we simply have to write any root vector in $\lie g(S)$ as a sum of Lie words in root vectors corresponding to real roots and apply the Chevalley involution. Hence all results obtained so far can be applied to them. Part $(3)$ of the next proposition is crucial and generalizes 
\cite[Proposition 4.1]{carbone2021commutator} (rank $2$ case) and 
\cite[Corollary 11.1.5]{roy2019maximal} (affine case).
\begin{prop}\label{corgen34}\begin{enumerate}
    \item For any real closed subroot system $\Psi$ we have $$\lie g(\Psi)=\lie g(\Pi(\Psi)).$$ 
    
    \item Let $\Sigma$ be a $\pi$-system of $\Delta$. Then $\Delta(\lie g(\Sigma))\cap \Delta^{\pm}$ is contained in the $\mathbb{Z}_{\pm}$-span of $\Sigma$. \vspace{0,15cm}
    
    \item For any real closed subroot system $\Psi$ we have $$\Psi=\Delta(\lie g(\Psi))^{\mathrm{re}}.$$
\end{enumerate}
\begin{proof} \begin{enumerate}[leftmargin=*]
    \item Since $\Pi(\Psi)\subseteq \Psi$ we only have to show that $\mathfrak{g}_\beta\subseteq \lie g(\Pi(\Psi))$ for all $\beta\in \Psi$. Given $\beta\in \Psi$ we choose $w\in W_{\Pi(\Psi)}$ such that $\beta=w\alpha$ for some $\alpha\in \Pi(\Psi).$ As in the proof of Lemma~\ref{baseisapisystem} we denote by $\ell_{\Pi(\Psi)}$ the length function on $W_{\Pi(\Psi)}$ with respect to $\Pi(\Psi).$ We shall prove by induction on $\ell_{\Pi(\Psi)}(w)$ that
$\mathfrak{g}_\beta\subseteq \lie g(\Pi(\Psi))$ and
\begin{equation}\label{liebracket}
    x_\beta= [x_{\gamma_1},[x_{\gamma_2},[\dots,[x_{\gamma_{k-1}},x_{\gamma_k}]]]] \ \ \text{ for some }\gamma_i\in \pm \Pi(\Psi).
\end{equation} 
If $\ell_{\Pi(\Psi)}(w)=0,$ there is nothing to show. So let $\gamma\in \Pi(\Psi)$ be such that $\ell_{\Pi(\Psi)}(s_\gamma w)<\ell_{\Pi(\Psi)}(w).$ By induction hypothesis we have that $\beta':=s_\gamma w(\alpha)\in \Delta(\lie g(\Pi(\Psi)))$ and $x_{\beta'}$ is of the form \ref{liebracket}. Since $\lie g(\Pi(\Psi))$ is a regular symmetric subalgebra of $\lie g$ and $\gamma,\beta'\in \Delta(\lie g(\Pi(\Psi)))^{\mathrm{re}}$, we can apply Proposition \ref{keyprop} to get $\beta\in \mathbb{S}(\gamma,\beta')\subseteq \Delta(\lie g(\Pi(\Psi)))$. The fact that $x_{\beta}$ is a Lie word in $\{x_{\gamma}: \gamma\in \pm\Sigma\}$ follows exactly in the same way as the proof of Proposition~\ref{keyprop}(2) by considering the module $M^{\lie g(\Pi(\Psi))}(\gamma,\beta')$. 
\item For the second part, using the Jacobi identity, we first note that the right normed Lie words in $\{x_{\gamma}: \gamma\in \Sigma\cup -\Sigma\}$ span $\lie g(\Sigma).$ Since 
$\Sigma$ is a $\pi$-system, we have $[x_\gamma, x_{-\gamma'}]=\delta_{\gamma, \gamma'}h_\gamma$ for $\gamma,\gamma'\in\Sigma$. Using this, a straightforward induction on the length of a Lie word shows that 
the root space $\lie g(\Sigma)_\beta$ for $\beta\in  \Delta(\lie g(\Sigma))\cap\Delta^{\pm}$ is spanned by the right normed Lie words 
$$[x_{\gamma_1},[x_{\gamma_2},[\dots,[x_{\gamma_{k-1}},x_{\gamma_k}]]]],\ \ \ \gamma_1,\dots,\gamma_k\in \pm \Sigma.$$
\item Define $\Psi':=\Delta(\lie g(\Psi))\cap \Delta^{\mathrm{re}}$ which is a real closed subroot system by Lemma~\ref{keyprop} and note that by part (1) we have $$\Pi(\Psi')\subseteq \Psi'=\Delta(\lie g(\Pi(\Psi)))\cap \Delta^{\mathrm{re}}\implies \Pi(\Psi')\subseteq \Delta(\lie g(\Pi(\Psi)))\cap \Delta^{+}.$$
Hence each element in $\Pi(\Psi')$ is in the $\mathbb{Z}_+$-span of $\Pi(\Psi)$ by part (2). However, writing $\beta=a_1\gamma_1+\cdots+a_k\gamma_k$ for some $\gamma_i\in \Pi(\Psi)\subseteq \Psi\subseteq\Psi'$ and $a_1,\dots,a_k\in\mathbb{Z}_+$ we can use the minimality of the elements in $\Pi(\Psi')$ to obtain $\beta=\gamma_i$ for some $i\in\{1,\dots,k\}$. This gives $\Pi(\Psi')\subseteq \Pi(\Psi)$. Therefore 
$$\Psi'= W_{\Pi(\Psi')}(\Pi(\Psi'))\subseteq W_{\Pi(\Psi)}(\Pi(\Psi))=\Psi.$$
\end{enumerate}
\end{proof}
\end{prop}

        %%%%%%%%%%%%%%%%%%%%%%%%%%%%%%%%%%%%%%%%%%%%%%%%%%%%%%%%%%%%%%%%%%%%%%%%%%%%%%%%%%%%%%%%%%%%%%%%%%%%%%%%%%%%%%%%%%%%%%%%%%%%%%%%%%%%%%%%%%%%%%%%%%%%%%%%%%%%%%%%%%%%%%%%%%%%%%%%%%%%
        %%%%%%%%%%%%%%%%%%%%%%%%%%%%%%%%%%%%%%%%%%%%%%%%%%%%%%%%%%%%%%%%%%%%%%%%%%%%%%%%%%%%%%%%%%%%%%%%%%%%%%%%%%%%%%%%%%%%%%%%%%%%%%%%%%%%%%%%%%%%%%%%%%%%%%%%%%%%%%%%%%%%%%%%%%%%%%%%%%%%
        %%%%%%%%%%%%%%%%%%%%%%%%%%%%%%%%%%%%%%%%%%%%%%%%%%%%%%%%%%%%%%%%%%%%%%%%%%%%%%%%%%%%%%%%%%%%%%%%%%%%%%%%%%%%%%%%%%%%%%%%%%%%%%%%%%%%%%%%%%%%%%%%%%%%%%%%%%%%%%%%%%%%%%%%%%%%%%%%%%%%
        %%%%%%%%%%%%%%%%%%%%%%%%%%%%%%%%%%%%%%%%%%%%%%%%%%%%%%%%%%%%%%%%%%%%%%%%%%%%%%%%%%%%%%%%%%%%%%%%%%%%%%%%%%%%%%%%%%%%%%%%%%%%%%%%%%%%%%%%%%%%%%%%%%%%%%%%%%%%%%%%%%%%%%%%%%%%%%%%%%%%
\section{Root generated subalgebras are of Kac-Moody type}\label{section4}
In this section we want to describe the class of root generated subalgebras of the form $\lie g(\Psi)$ for real closed subroot systems $\Psi$. The main result is stated in Theorem~\ref{mainresbij}.
\subsection{} We shall recall the construction of the algebra $\lie{g}'(A)$ from \cite[Remark 1.5]{kac1990infinite} for a possibly infinite generalized Cartan matrix $A$ and derive some of the properties of the structure of $\lie{g}'(A).$ Let $I$ be the index set of $A$ and denote by $\tilde{\lie{g}}'(A)$ the Lie algebra generated by $e_i,f_i,\alpha_i^\vee,\ i\in I$ with relations $$[\alpha_i^\vee,\alpha_j^\vee], \ \,[e_i,f_j]-\delta_{ij}\alpha_i^\vee,\ \ [\alpha_i^\vee,e_j]-a_{ij}e_j, \ \ [\alpha_i^\vee,f_j]+a_{ij}f_j.$$ Let $Q$ be the free abelian group generated by $\alpha_i,i\in I$ and note that $\tilde{\lie g}'(A)$ admits a $Q$-grading with
    $$\deg(e_i)=\alpha_i,\ \deg(f_i)=-\alpha_i,\ \deg(\alpha_i^{\vee})=0.$$ There exists a unique maximal $Q$-graded ideal $\lie i$ of $\Tilde{\lie g}'(A)$ which intersects $\lie{h}'(A):=\sum_{i\in I} \bc\alpha_i^\vee$ trivially and set  $$\lie{g}'(A):=\tilde{\lie g}'(A)/ \lie{i}.$$ 
\begin{rem}
If the matrix $A$ is finite, then the above constructed Lie algebra $\lie g'(A)$ is simply the derived algebra of $\lie g(A)$ and the notation is consistent with the previous sections.
\end{rem}

The next lemma is proven similarly as \cite[Proposition 1.6/1.7]{kac1990infinite}. 
    \begin{lem}\label{lemidealkacmoody}
        Let $A$ be a possibly infinite generalized Cartan matrix.
        \begin{enumerate}
            \item  The centre $\lie c'(A)$ of $\lie g'(A)$ is given by $$\lie c'(A)=\{h\in \lie h'(A): \alpha_i(h)=0\ \ \forall i\in I\}.$$
            \item If $A$ is indecomposable, then any proper $Q$-graded ideal $\lie i$ of $\lie{g}'(A)$ is contained in the centre $\lie c'(A).$
             \vspace{0,2cm}
         \item Let $A=\bigoplus A_k$ be a decomposition of $A$ into indecomposable generalied Cartan  matrices. Then for any ideal (resp. $Q$-graded ideal) $\lie{i}$ of $\lie{g}'(A)$, we have  $\lie{i}=\bigoplus \lie{i}_k$ for some ideals (resp. $Q$-graded ideals) $\lie{i}_k$ of $\lie{g}'(A_k)$. 
        \end{enumerate} \qed
    \end{lem}
 We need the following definition in order to state our main theorem.
\begin{defn}\label{defnKMtype}
      A subalgebra $\lie a$ of a Kac-Moody Lie algebra $\lie g$ is called of \textit{Kac-Moody type} if there exists a symmetrizable (possibly infinite) generalized Cartan matrix $A$ and a short exact sequence of Lie algebra homomorphisms
        $$0\longrightarrow \mathrm{ker}(\varphi)\longrightarrow \lie g'(A)\stackrel{\varphi}{\longrightarrow} \lie a\longrightarrow 0$$
such that $\mathrm{ker}(\varphi)\subseteq \lie c'(A)$.
    \end{defn}
Our main theorem whose proof will be given in the rest of this section is the following.  The second correspondence for affine Lie algebras has been proved in \cite{roy2019maximal}. 
    \begin{thm}\label{mainresbij}
     Let $\lie g$ be a Kac-Moody algebra. We have the following bijections \vspace{10pt}
    $$\begin{array}{cccccr}
     \vspace{10pt}\left\{\begin{array}{c}
            \pi-\text{systems of }\Delta  \\
             \text{contained in } \Delta^+
       \end{array}\right\}  & \longleftrightarrow & \left\{\begin{array}{c}
            \text{real closed subroot}  \\
             \text{systems of } \Delta
       \end{array}\right\}& \longleftrightarrow & \left\{\begin{array}{c}
            \text{root generated}\\
             \text{subalgebras of $\lie g$} 
       \end{array}\right\} \\ \vspace{10pt}
        \Sigma  & \longmapsto & W_\Sigma(\Sigma)\\
    \Pi(\Psi) & \longmapsfrom &\Psi & \longmapsto& \lie g(\Psi)&\\ \\
    &&\Delta(\lie g(S))^{\mathrm{re}}& \longmapsfrom &\lie g(S)&
    \end{array}$$
    
    Moreover, any root generated subalgebra is of Kac-Moody type. 
    \end{thm}
 
\subsection{}  This subsection is devoted to the first correspondence in Theorem~\ref{mainresbij}.  The map $\Psi\mapsto \Pi(\Psi)$ is well-defined by Lemma~\ref{baseisapisystem}. Now let $\Sigma\subseteq \Delta^+$ be a $\pi$-system of $\Delta$ and note that $\Psi':=\Delta(\lie g(\Sigma))^{\mathrm{re}}$ is a real closed subroot system of $\Delta$ by Proposition~\ref{keyprop}. Moreover, since $\Pi(\Psi')\subseteq \Delta(\lie g(\Sigma))\cap \Delta^+$ each element of $\Pi(\Psi')$ is in the $\mathbb{Z}_+$-span of $\Sigma$ by Lemma~\ref{subal1}. Hence  we have $\Pi(\Psi')\subseteq \Sigma$ since $\Pi(\Psi')$ is the set of minimal elements (c.f. Section \ref{baseisapisystem}). This implies
$$\Psi'\subseteq W_{\Pi(\Psi')}(\Pi(\Psi'))\subseteq W_{\Sigma}(\Sigma)\subseteq W_{\Psi'}(\Psi')\subseteq \Psi'.$$
It follows that $W_{\Sigma}(\Sigma)$ is a  real closed subroot system of $\Delta$ and the map $\Sigma\mapsto W_{\Sigma}(\Sigma)$ is also well-defined. The fact that the maps are inverses of each other follows from Proposition \ref{uniqq}. 
\subsection{} The maps in the second correspondence of Theorem~\ref{mainresbij} are well-defined by Proposition~\ref{keyprop} and the fact that root generated subalgebras are symmetric regular subalgebras. Next we show 
\begin{equation}\label{uuhz41}\lie g(S)=\lie g(\Delta(\lie g(S))^{\mathrm{re}})\end{equation}
where the inclusion $\lie g(\Delta(\lie g(S))^{\mathrm{re}})\subseteq \lie g(S)$ is obvious. However, by definition we have $S\subseteq \Delta(\lie g(S))^{\mathrm{re}}$ and therefore the reverse inclusion also holds which finally gives equality. Now the bijective correspondence follows from Proposition~\ref{corgen34}(3). 

\subsection{} It remains to prove the last statement of the theorem. Note that we can not use the results 
     of \cite{carbone2021varvec} since $\Pi(\Psi)$ is in general neither finite nor linearly independent.\par For a real closed subroot system $\Psi$ let $\Sigma:=\Pi(\Psi)=\{\beta_i : i\in I\}\subseteq \Psi\cap \Delta^+$ be the unique $\pi$-system of $\Delta$ (c.f. Proposition~\ref{uniqq}). Define the (possibly infinite) matrix $B_{\Sigma}=(b_{ij})$ by $b_{ij}=(\beta_j,\beta_i^\vee)$ which satisfies $b_{i,j}\leq 0$ since $\beta_i-\beta_j\notin \Delta$. In particular, $B_{\Sigma}$ is a  generalized Cartan matrix and we denote the Chevalley generators of the Lie algebra $\lie{g}'(B_{\Sigma})$ by $e_i,f_i,\alpha^{\vee}_i, i\in I$. We define a map 
          \begin{equation}\label{mapphi}
          \varphi:\lie{g}'(B_{\Sigma})\to \lie{g}(\Psi),\ \ e_i\mapsto x_{\beta_i},\ \ \ f_i\mapsto x_{-\beta_i},\ \ \ \alpha^{\vee}_i\mapsto \beta_i^\vee.
          \end{equation}
 It is easy to see that all the relations of $\widetilde{\lie g}'(B_{\Sigma})$ are satisfied in $\lie{g}(\Psi)$ and the unique $Q$-graded ideal in $\widetilde{\lie g}'(B_{\Sigma})$ is mapped to zero as it is generated by the Serre relations \cite[Remark 2.1]{jurisich1998generalized}. Thus $\varphi$ is a well-defined Lie algebra homomorphism. Since $\lie g(\Sigma)=\lie g(\Psi)$ by Proposition~\ref{corgen34} we also have that $\varphi$ is surjective.
 
 \medskip 
\begin{defn}For $\underline{k}\in\bz^I$ we define $$\alpha(\underline{k})=\sum k_i\alpha_i,\ \  \underline{k}\boldsymbol{\cdot} \Sigma:=\sum k_i\beta_i$$
 and for $\beta\in \Delta(\lie g(\Psi))\cap\Delta^{\pm}$ we set
$$T_{\beta}:=\left\{\underline{k}\in \bz_{\pm}^I: \beta=\underline{k}\boldsymbol{\cdot}\Sigma\right\}.$$
\end{defn}
Note that $T_{\beta}\neq \emptyset$ by Proposition~\ref{corgen34}(2) and we can have $|T_{\beta}|\geq 2$ by Example~\ref{FN2d}. We collect some facts.
\begin{itemize}
    \item  We have $|T_{\beta_i}|=1$ for all $i\in I.$ To see this, write $\beta_i=\sum k_j\beta_j$ such that $k_j\geq 0$ for all $j$ and suppose there exists $r\neq i$ with $k_r>0.$ If $k_i=0,$ then we have $$0<(\beta_i,\beta_i)=\sum k_j(\beta_j,\beta_i)\le 0$$ since $(\beta_j,\beta_i)\le 0$ for all $j\neq i,$ a contradiction. If $k_i\geq 1,$ then $\sum_j (k_j-\delta_{ij})\beta_j=0$ with all coefficients non-negative and $k_r>0$ which is also impossible since $\Sigma\subseteq \Delta^+$.

    \vspace{0,2cm}
    
    \item If $\Pi(\Psi)$ is linearly independent, then $|T_\beta|=1$ for all $\beta\in \Delta(\lie g(\Psi)).$
\end{itemize}
The surjectivity of $\varphi$ gives 
\begin{equation}\label{sumoftuples}
    \lie g(\Psi)_\beta=\sum_{\substack{\underline{k}\in T_\beta\\ \alpha(\underline{k})\in \Delta(\lie g'(B_{\Sigma}))}} \lie g(\Psi)_{\underline{k}},\ \ \ \lie g(\Psi)_{\underline{k}}:=\varphi(\lie g'(B_{\Sigma})_{\alpha(\underline{k})}).
\end{equation}
for all $\beta\in \Delta(\lie g(\Psi))$. We shall prove in fact that the above sum above is direct. We need the following elementary lemma. 
\begin{lem}\label{nulloperation1} Let $\alpha\in \Delta(\lie g'(B_{\Sigma}))$. 
\begin{enumerate}
    \item 
    The restriction $\varphi_\alpha:\lie g'(B_{\Sigma})_\alpha\to \lie g(\Psi)$ of the map \eqref{mapphi} to $\lie g'(B_{\Sigma})_\alpha$ is  injective. 
    \vspace{0,2cm}
    
    \item   Let $\underline{k}\in \bz_{\pm}^I$ such that $\alpha=\alpha(\underline{k})$ and $y\in \lie g(\Psi)_{\underline{k}}$ be such that $[x_{\mp\beta_i},y]=0$ for all $i\in I.$ Then we have $y=0.$
\end{enumerate}

\end{lem}
\begin{proof}
The first part is proven similarly as \cite[Corollary 2.11]{carbone2021varvec} and the second part as \cite[Lemma 1.5]{kac1990infinite}. Nevertheless, we will prove the second part for the readers convenience for tuples $\underline{k}\in \bz_{+}^I$. Let $y\in \lie g(\Psi)_{\underline{k}}$ such that $[x_{-\beta_i},y]=0$ for all $i\in I.$ Since the map
    $\varphi_\alpha:\lie g'(B_{\Sigma})_{\alpha(\underline{k})}\to \lie g(\Psi)_{\underline{k}}$ is a linear isomorphism, there exists a $Q$-homogeneous element $x\in \lie g'(B_{\Sigma})_{\alpha(\underline{k})}$ such that $\varphi(x)=y$. So it will be enough to show in the rest of the proof that $x=0$. If $[f_i,x]\neq 0$ for some $i\in I$ we obtain once more with the injectivity of $\varphi_{\alpha(\underline{k})-\alpha_i}$:
    $$[f_i,x]\in \lie g'(B_{\Sigma})_{\alpha(\underline{k})-\alpha_i}\backslash\{0\}\implies [x_{-\beta_i},y]=\varphi_{\alpha(\underline{k})-\alpha_i}([f_i,x])\neq 0.$$
This is a contradiction and hence $[f_i,x]=0$ for all $i\in I$. In particular, the ideal of $\lie g'(B_{\Sigma})$ generated by $x$ is $Q$-graded and intersects $\lie h'(B_{\Sigma})$ trivially. This is impossible unless $x=0$. 
\end{proof}
\begin{prop}\label{sumisdirect1}
    The sum in \ref{sumoftuples} is a direct sum of vector spaces.
\end{prop}
\begin{proof} Let $\beta\in \Delta(\lie g(\Psi))\cap \Delta^{+}$ and let $\mathrm{ht}(\beta)$ be the sum over the coefficients when $\beta$ is expressed in terms of a simple system of $\lie g(A)$. The proof is by induction on $\mathrm{ht}(\beta)$ where the base case ($\beta=\beta_i$ for some $i\in I$) obviously holds by the discussion preceeding Lemma~\ref{nulloperation1}. Let $\underline{k}_1,\dots, \underline{k}_r$ be distinct elements of $T_\beta$ such that $\alpha(\underline{k}_1),\dots,\alpha(\underline{k}_r)$ are all roots of $\lie g'(B_{\Sigma})$. By contradiction assume
$$y_1+\cdots +y_r=0,\ \ y_i\in \lie g(\Psi)_{\underline{k}_i}\backslash\{0\}.$$
Since $\underline{k}_1\in \mathbb{Z}^I_+$ and $\alpha(\underline{k}_1)\in \Delta(\lie g'(B_{\Sigma}))$ we can find with Lemma~\ref{nulloperation1}(2) an index $j\in I$ such that $[x_{-\beta_j},y_1]\neq 0$; in particular $\beta-\beta_j\in \Delta(\lie g(\Psi))\cap \Delta^+$ since the $j$-th entry of $\underline{k}_1$ has to be positive. We get 
$$[x_{-\beta_j},y_1]+\cdots +[x_{-\beta_j},y_r]=[x_{-\beta_j},y_{i_1}]+\cdots+[x_{-\beta_j},y_{i_s}]=0$$
where $\{i_1,\dots,i_s\}\subseteq \{1,\dots,r\}$ is the collection of indices with $[x_{-\beta_j},y_{i_1}],\dots,[x_{-\beta_j},y_{i_s}]\neq 0$ and $1\in\{i_1,\dots,i_s\}$. Let $\underline{e}_j$ be the tuple with entry $1$ at $j$-th position and zero elsewhere. Since $\underline{k}_{i_1}-\underline{e}_j,\dots,\underline{k}_{i_s}-\underline{e}_j\in T_{\beta-\beta_j}$ are pairwise distinct elements and $\alpha(\underline{k}_{i_1}-\underline{e}_j),\dots,\alpha(\underline{k}_{i_s}-\underline{e}_j)$ are all roots of $\lie g'(B_{\Sigma})$, we can apply the induction hypothesis and get $[x_{-\beta_j},y_1]=0$ which is a contradiction.
This completes the proof.
\end{proof}

\begin{cor}\label{kerqgraded}
The kernel of the map $\varphi$ from \eqref{mapphi} is a $Q$-graded ideal in $\lie g'(B_{\Sigma}).$

\begin{proof} 
    Let $x\in\ker(\varphi)$ and $x=x_1+x_2+\cdots +x_r$ be the decomposition of $x$ into $Q$-homogeneous components. We will prove the statement by induction on the length of such a decomposition. The case $r=1$ is trivial; so let $r>1$. Since no partial sum is contained in $\ker(\varphi)$ (otherwise we are done by induction) there exists $\beta\in \Delta(\lie g(\Psi))$ with $\varphi(x_i)\in \lie g(\Psi)_{\beta}$. Since $x_i$ is $Q$-homogeneous we can find distinct $\underline{k}_1,\dots ,\underline{k}_r\in T_{\beta}$ with $\varphi(x_i)\in \lie g(\Psi)_{\underline{k}_i}$ and $\alpha(\underline{k}_1),\dots\alpha(\underline{k}_r)\in \Delta(\lie g'(B_{\Sigma}))$. Since 
    $$\varphi(x)=0=\varphi(x_1)+\cdots+\varphi(x_r)$$
    we get with Proposition~\ref{sumisdirect1} that this is only possible if $x_i\in \ker(\varphi)$ for all $i\in\{1,\dots,r\}$.
  \end{proof}
 
\end{cor}

The following theorem together with Proposition~\ref{keyprop} and \eqref{uuhz41} finishes the proof of Theorem~\ref{mainresbij}.
   \begin{thm}\label{KMtypereg}
    Let $\Psi$ be a real closed subroot system of $\Delta$, then $\lie g(\Psi)$ is of Kac-Moody type. 
          \begin{proof} Let $\varphi$ be the surjective Lie algebra homomorphism from \eqref{mapphi}.
         Assume first that $B_{\Sigma}$ is indecomposable. Then by Lemma~\ref{lemidealkacmoody}(2) and Corollary~\ref{kerqgraded} we get $\ker(\varphi)\subseteq \lie c'(B_{\Sigma})$
         and hence the desired short exact sequence.
Now assume that $B_{\Sigma}$ is decomposable and let $B_{\Sigma}=\bigoplus B_{\Sigma,k}$ be a decomposition into indecomposables. Then we have by 
Lemma~\ref{lemidealkacmoody}(3) and Corollary~\ref{kerqgraded} that $\ker(\varphi)=\bigoplus\lie{i}_k$ for some $Q$-graded proper ideals $\lie{i}_k$ of $\lie{g}'(B_{\Sigma,k})$. Again by Lemma~\ref{lemidealkacmoody}(2) we get $\lie{i}_k\subseteq \lie c'(B_{\Sigma,k})$ and thus
$$\ker(\varphi)=\bigoplus\lie{i}_k\subseteq \bigoplus \lie c'(B_{\Sigma,k})\subseteq \lie c'(B_{\Sigma}).$$
\end{proof}
\end{thm}
 \begin{rem}\label{remlininde1} If $\Sigma=\Pi(\Psi)$ ($\Psi$ real closed subroot system) is linearly independent, then  $\lie g(\Psi)$ is in fact isomorphic to $\lie g'(B_{\Sigma})$. The linear independence of $\Sigma$ is guaranteed for example when $\lie g$ is of finite type or $B_\Sigma$ is an affine GCM. 
    \end{rem}
\subsection{} We end this section with another example. 
\begin{example}\label{exrank2}
    Let $A=\begin{pmatrix}
        2 & -3\\ -3 & 2
    \end{pmatrix}$ and note that this is a hyperbolic GCM of rank $2$. Let $$\beta_1^j:=f_{2j}\alpha_1+f_{2j+2}\alpha_2,\ \ \beta_2^j:=f_{2j+2}\alpha_1+f_{2j}\alpha_2, \text{ for }j\in\bz_+,$$ where $f_j$ is the $j$-th Fibonacci number:
    $$f_0=0,\ \ f_1=1,\ \ f_{j+2}=f_{j+1}+f_j, \text{ for }j\in\bz_+.$$
    The positive real roots of $\lie g(A)$ are given by $\Delta(\lie g(A))^{\mathrm{re},+}=\{\beta_1^j,\beta_2^j:j\in\bz_+\}$ (see \cite[Exercise 5.28]{kac1990infinite}). We claim that 
    $$\beta_1^j-\beta_1^k,\ \beta_2^j-\beta_2^k\in \Delta(\lie g(A))^{\mathrm{im}},\ \ \text{ for all } j,k\in\bz_+ \text{ with } j\neq k,$$
    $$\beta_1^j-\beta_2^k\notin \Delta(\lie g(A)),\ \ \text{ for all } j,k\in \bz_+.$$ Assuming the claim, it follows that every $\pi$-system $\Sigma$ of $\Delta$ with $\Sigma\subseteq \Delta^+$ is linearly independent and satisfies $|\Sigma|\le 2.$ For the claim, using Lemma \ref{sumnotreal} and \cite[Proposition 5.10(c)] {kac1990infinite}, it will be enough to show $(\beta,\beta)\le 0$ if $\beta=\beta_1^j-\beta_1^k$ or $\beta=\beta_2^j-\beta_2^k$ for some $j\neq k\in\bz_+$ and $(\beta,\beta)> 0$ if $\beta=\beta_1^j-\beta_2^k$ for some $j,k\in\bz_+.$ Let $\beta=\beta_1^j-\beta_2^k.$ Then 
    \begin{align*}
        (\beta,\beta)&=2\left(f_{2j+1}+f_{2k+1})^2-(f_{2j}-f_{2k+2})(f_{2j+2}-f_{2k})\right)\\
        &=(f_{2j+1}^2-f_{2j}f_{2j+2})+(f_{2k+1}^2-f_{2k+2}f_{2k})+2f_{2j+1}f_{2k+1}+f_{2j}f_{2k}+f_{2k+2}f_{2j+2}\\
        & >0 \ \ \ \ \text{ (by Cassini's formula, \cite[Theorem 5.3]{koshy2001fibonacci})}. 
    \end{align*}
    
    If $\beta=\beta_1^j-\beta_1^k$ (similarly $\beta=\beta_2^j-\beta_2^k$), then 
    \begin{align*}
        \frac{1}{2}(\beta,\beta)= &\left((f_{2k+1}-f_{2j+1})^2-(f_{2j}-f_{2k})(f_{2j+2}-f_{2k+2})\right)\\
        =&\left(2-(f_{2j+1}f_{2k+1}-f_{2j}f_{2k+2}+f_{2j+1}f_{2k+1}-f_{2j+2}f_{2k})\right)\\
        =& \left(2-(f_{2k-2j+1}+f_{2k-2j+1})\right) \ \ \text{(by \cite[Indentity 2, page 87]{koshy2001fibonacci})}\\
        \leq &0
    \end{align*}
Hence by Remark~\ref{remlininde1}, all root generated subalgebras of $\lie g(A)$ are isomorphic to the derived algebra of a Kac-Moody algebra.
    \end{example}
    
\section{Symmetric regular subalgebras of untwisted affine Kac-Moody algebras}\label{section5}
We have seen that root generated subalgebras form an important class of symmetric regular subalgebras and that they correspond to real closed subroot systems. In the rest of this section let $\lie g$ be an untwisted affine Kac-Moody algebra. We will use the explicit description of (maximal) real closed subroot systems from \cite{roy2019maximal} to classify all symmetric regular subalgebras in this section.  
\subsection{}
Let $\mathring{\lie g}$ be a finite dimensional simple Lie algebra with Cartan subalgebra $\mathring{\lie h}$ and set of roots $\mathring{\Delta}$. The untwisted affine algebra is realized as $\lie g=\mathring{\lie g}\otimes \mathbb{C}[t^{\pm}]\oplus \mathbb{C}c\oplus \mathbb{C}d$ and the roots with respect to the Cartan subalgebra $\lie h=\mathring{\lie h}\oplus \mathbb{C}c\oplus \mathbb{C}d$ are given by 
$$\Delta=\Delta^{\mathrm{re}}\cup \Delta^{\mathrm{im}},$$
$$\Delta^{\mathrm{re}}=\{\alpha + r\delta : \alpha\in\mathring{\Delta},\ r\in \mathbb{Z}\},\ \ \Delta^{\mathrm{im}}=\{r\delta : r\in \mathbb{Z}\backslash \{0\}\}.$$
Moreover, we have
$$x_{\alpha+r\delta}=\mathbb{C}(x_{\alpha}\otimes t^r),\ \alpha\in \mathring{\Delta}, \ r\in\mathbb{Z},\  \ \lie g_{r\delta}=\mathring{\lie h}\otimes t^r,\ r\neq 0.$$
Given a symmetric regular subalgebra $\lie s$ we set in the rest of this subsection $\Psi:=\Delta(\lie s)^{\mathrm{re}}$ which is a real closed subroot system of $\Delta$. Hence from \cite[Section 2]{roy2019maximal} we can derive that there exists a closed subroot system $\mathring{\Psi}\subseteq\mathring{\Delta}$ 
with irreducible components $\mathring{\Psi}_1,\dots,\mathring{\Psi}_s$,
non-negative integers $k_1,\dots,k_s$ and
$\mathbb{Z}$-linear functions $f_i: \mathring{\Psi}_i \to \mathbb{Z} $, $1\leq i\leq s$, such that
\begin{equation}\label{defPsi}
    \Psi = \bigcup_{i=1}^s \Psi_i,\ \ \Psi_i=\left\{\alpha + (f_i(\alpha) + k_ir)\delta : \alpha\in \mathring{\Psi}_i, r\in \mathbb{Z}\right\}.
\end{equation}
\begin{rem}
For the explicit description of closed subroot systems of finite root system see \cite{dynkin1952semisimple}.
\end{rem} 
Let $\lie h(\mathring{\Psi}_i)$ be the subspace of $\mathring{\lie h}$ spanned by $\{h_\alpha : \alpha \in \mathring{\Psi}_i\}$ and denote the orthogonal complement of $\lie h(\mathring{\Psi}_i)$ in $\mathring{\lie h}$ with respect to the Killing form by
$\lie h(\mathring{\Psi}_i)^\perp$. 
Then the subalgebra $\lie g(\Psi)$ of $\lie s$ is given by 
\begin{equation}\label{rootgenaff}\lie g(\Psi)= \bigoplus_{i=1}^s\bigoplus\limits_{\substack{\alpha\in \mathring{\Psi}_i\\ r\in \mathbb{Z}}} \mathbb{C}(x_{\alpha}\otimes t^{f_i(\alpha)+k_ir})\oplus \bigoplus\limits_{\substack{x\in\mathbb{Z}}}\left(\bigoplus_{\substack{i=1\\ x\in k_i\mathbb{Z}}}^s\lie h(\mathring{\Psi}_i)\right)\otimes t^{x} \oplus \mathbb{C}c'.\end{equation}
where $$c'=\begin{cases}
  0  & \text{if }k_i=f_i(\alpha)=0\ \ \forall \alpha\in\mathring{\Psi}_i,\ i=1,2,\dots,s,\\
 c & \text{otherwise}.
  \end{cases}$$
To see this, we only need to observe that the right hand side of \eqref{rootgenaff} is a subalgebra containing $\lie g_{\alpha}$ for all $\alpha\in \Psi$; the containment of the right hand side in $\lie g(\Psi)$ is obvious. In order to understand the structure of $\lie s$ completely we study further the imaginary root spaces that appear in $\lie s$.
We define $I(\lie s):=\{r\in\mathbb{Z} : r\delta\in \Delta(\lie s)^{\mathrm{im}}\}\cup\{0\}$
and write
$$I(\lie s)=(k_1\mathbb{Z} \cup \cdots \cup k_s\mathbb{Z})\ \dot\cup \ \Lambda$$
for some subset $\Lambda\subseteq \mathbb{Z}\backslash \{0\}$ with $\Lambda=-\Lambda$. The latter condition is implied by the symmetry of $\Delta(\lie s)^{\mathrm{im}}$. For each $x\in I(\lie s)$, we have $\lie s_{x\delta} = \mathring{\lie h}_x\otimes t^{x}$ for some subspace $\mathring{\lie h}_x\subseteq \mathring{\lie h}$ and $\lie h(\mathring{\Psi}_i)\subseteq \mathring{\lie h}_x$ if $x\in k_i\mathbb{Z}$. Moreover, since $\lie s$ is a subalgebra, we obtain for arbitrary $v\in \mathring{\lie h}_x$ and $\alpha\in \mathring{\Psi}_i$ 
$$\left[x_\alpha\otimes t^{f_i(\alpha)+k_ir}, v\otimes t^x\right]=\alpha(v)(x_\alpha\otimes t^{f_i(\alpha)+k_ir+x})\in \lie s.$$
If $x\notin k_i\mathbb{Z}$, this is only possible if $\alpha(v)=0$ and thus 
 $$\mathring{\lie h}_x\subseteq \bigcap_{\substack{i=1\\ x\notin k_i\mathbb{Z}}}^s \lie h(\mathring{\Psi}_i)^\perp, \ \ \ \forall x\in I(\lie s).$$
\begin{thm}\label{mainthmaffine1} Let $\lie s$ be a symmetric regular subalgebra of $\lie g$ such that $\lie s\subseteq [\lie g,\lie g]$ and set $\Psi=\Delta(\lie s)^{\mathrm{re}}$. Then there exists:
\begin{itemize}
    \item a closed subroot system $\mathring{\Psi}\subseteq \mathring{\Delta}$ with irreducible components $\mathring{\Psi}_1,\dots, \mathring{\Psi}_s$ and non-negative integers $k_1,\dots,k_s\in\mathbb{Z}_+$,
    \vspace{0,1cm}
    
    \item a symmetric subset $\Lambda\subseteq \mathbb{Z}\backslash\{0\}$ satisfying $\Lambda\cap k_i\mathbb{Z}=\emptyset$ for all $1\leq i\leq s$,
    \vspace{0,1cm}
    \item $\mathbb{Z}$-linear functions $f_i: \mathring{\Psi}_i \to \mathbb{Z} $, $1\leq i\leq s$, \vspace{0,1cm}
    
    \item vector subspaces $V_x\subseteq \mathring{\lie h}$ for all $x\in \mathrm{I}(\underline{\mathbf k},\Lambda):=(k_1\mathbb{Z} \cup \cdots \cup k_s\mathbb{Z})\cup \Lambda$ satisfying 
    
    $$V_x\subseteq \displaystyle\bigcap_{\substack{i=1\\  x\notin k_i\mathbb{Z}}}^s \lie h(\mathring{\Psi}_i)^\perp,\ \ \ V_{x}\cap \left(\sum_{\substack{i=1\\ x\in k_i\mathbb{Z}}}^s\lie h(\mathring{\Psi}_i)\right)=0.$$ 
\end{itemize}
such that
\begin{equation}\label{symmform4}\lie s = \lie g(\Psi)\oplus \bigoplus_{x\in \mathrm{I}(\underline{\mathbf k},\Lambda)} V_x\otimes t^x\end{equation}
where $\lie g(\Psi)$ is given as in \eqref{rootgenaff}. Conversely any subalgebra of the form \eqref{symmform4} is a symmetric regular subalgebra contained in the derived algebra $[\lie g,\lie g]$. 
\begin{proof}
The forward direction follows from the discussion preceeding the theorem. So we only check the converse part, namely that the right hand side of \eqref{symmform4} is indeed a symmetric regular subalgebra. Define the right hand side of \ref{symmform4} as $\lie s'$ and note that $\lie s'$ is clearly $\lie h$-invariant. To check that it is a Lie algebra we let $x\in \mathrm{I}(\underline{\mathbf k},\Lambda)$ and $\alpha\in \mathring{\Psi}_i$, $1\le i\le s$. Consider 
\begin{equation}\label{liealgnachweis}\left[x_\alpha\otimes t^{f_i(\alpha)+k_ir}, v\otimes t^x\right]=\alpha(v)(x_\alpha\otimes t^{f_i(\alpha)+k_ir+x}),\ \ v\in V_x.\end{equation}
If $x\notin k_i\mathbb{Z}$ we have $V_x\subseteq \lie h(\mathring{\Psi}_i)^\perp$ by construction and thus $\alpha(v)=0$. If $x\in k_i\mathbb{Z}$ then we have 
$\alpha(v)(x_\alpha\otimes t^{f_i(\alpha)+k_ir+x})\in \lie g(\Psi)\subseteq \lie s'.$ In either case \eqref{liealgnachweis} is contained in $\lie s'$ again. All other cases are checked similarly. 
\end{proof}
\end{thm}

\begin{rem}\label{symmremark} 
\begin{enumerate}[leftmargin=*]
\item A symmetric regular subalgebra is contained in the derived algebra $[\lie g,\lie g]$ or is of the form $\lie s'\oplus \mathbb{C}d$ for some symmetric regular subalgebra $\lie s'\subseteq [\lie g,\lie g]$.
    \item If $k=k_1=\cdots =k_s$, then we have 
    $$\mathring{\lie h}=\displaystyle\bigcap_{\substack{j=1\\ x\notin k_j \mathbb{Z}}}^s \lie h(\mathring{\Psi}_j)^\perp,\ \forall x\in k\mathbb{Z}$$
   and $V_x$ for $x\in k\mathbb{Z}$ can be chosen to be any subspace of $\mathring{\lie h}$ satisfying $V_{x}\cap \lie h(\mathring{\Psi})=0$. 
   \item If $\Delta(\lie s)\cap \mathring{\Delta}=\mathring{\Delta}$, then $s=1$, $V_x=0$ for all $x$ and $\lie s=\lie g(\Delta(\lie s)^{\mathrm{re}})$. This follows immediately from Theorem~\ref{mainthmaffine1}.
    
\end{enumerate}    
\end{rem}
    We denote by $\lie s(\mathring{\Psi}, \underline{\mathbf k},\Lambda,(f_i), (V_x))$ the 
symmetric regular subalgebra in \eqref{symmform4} associated to the tuple $(\mathring{\Psi}, \underline{\mathbf k},\Lambda,(f_i), (V_x))$. 
\subsection{} We call a proper symmetric regular subalgebra \textit{maximal} if it is not properly contained in any other proper symmetric regular subalgebra. From Theorem ~\ref{mainthmaffine1} we can easily determine the subclass of maximal symmetric regular subalgebras of $\lie g$. It is not surprising that these are connected to maximal real closed subroot systems $\Psi$ which are of the following form (see \cite[Section 2]{roy2019maximal}).  
\begin{enumerate}
    \item[Case 1.]
   There exists a prime number $k$ and a $\mathbb{Z}$-linear function $f:\mathring{\Delta}\to \mathbb{Z}$ such that $$\Psi=(\mathring{\Delta}, k, f):=\{\alpha + (f(\alpha)+kr)\delta : \alpha\in \mathring{\Delta},\ r\in \mathbb{Z}\}.$$
   % In this case, we simply denote $\Psi = (\mathring{\Delta}, k, f)$.
    \item[Case 2.] There exists a proper maximal closed subroot system $\mathring{\Psi}\subsetneq\mathring{\Delta}$ such that $$\Psi=(\mathring{\Psi}, 1, \mathbf{0}):=\{\alpha + r\delta : \alpha\in \mathring{\Psi}, r\in \mathbb{Z}\}.$$
     % Again in this case, we simply denote $\Psi = (\mathring{\Psi}, 1, \mathbf{0})$.
\end{enumerate}
Now we are ready to state the classification of maximal symmetric regular subalgebras. 
\begin{thm}\label{maintheoremaffinemaximalones}
Let $\lie g$ be an untwisted affine Kac-Moody algebra. Then $\lie s$ is a maximal symmetric regular subalgebra if and only if one of the following conditions hold
\begin{itemize}
    \item $\lie s=[\lie g,\lie g]$,
    \vspace{0,1cm}
    \item there exists a maximal closed subroot system $\mathring{\Psi}\subsetneq \mathring{\Delta}$ with
$$\lie s=\lie h\oplus \bigoplus\limits_{\substack{\alpha\in \mathring{\Psi}\\ r\in \mathbb{Z}}}(\mathbb{C}x_\alpha\otimes t^r)\oplus \bigoplus\limits_{r\in \mathbb Z\backslash\{0\}} (\mathring{\lie h}\otimes t^r)$$
\item there exists a prime number $k$ and a $\mathbb{Z}$-linear function $f:\mathring{\Delta}\rightarrow \mathbb{Z}$ with
$$\lie s= \lie s(\mathring{\Delta}, k, \emptyset, f, (0))\oplus \mathbb{C}d.$$
\end{itemize}
\end{thm}
\begin{proof}
Let $\lie s$ be a maximal symmetric regular subalgebra. From Theorem~\ref{mainthmaffine1} and Remark~\ref{symmremark} we have $\lie s=\lie s(\mathring{\Psi}, \underline{\mathbf k},\Lambda,(f_i), (V_x))\oplus \mathbb{C} d$ for a suitable choice of data as in Theorem~\ref{mainthmaffine1}.

\textit{Case 1}: Suppose that $\mathring{\Psi}\subsetneq \mathring{\Delta}$. In this case we can choose a maximal closed subroot system $\mathring{\Psi}'\supseteq\mathring{\Psi}$
and set
$$\lie s':=\lie s(\mathring{\Psi}', \underline{\mathbf 1},\emptyset ,(0), (\lie h(\mathring{\Psi})^{\perp}))\oplus \mathbb{C} d=\lie h\oplus \bigoplus\limits_{\substack{\alpha\in \mathring{\Psi}'\\ r\in \mathbb{Z}}} \mathbb{C}(x_{\alpha}\otimes t^{r})\oplus \bigoplus_{r\in \mathbb{Z}\backslash\{0\}} (\mathring{\lie h}\otimes t^{r})\oplus \bc d.$$
Clearly, $\lie s'$ is a symmetric regular subalgebra with $\lie s\subseteq \lie s'\subsetneq \lie g$. By the maximality we must have $\lie s=\lie s'$ and $\mathring{\Psi}=\mathring{\Psi}'$ is a maximal closed subroot system. \medskip 

Conversely, any $\lie s(\mathring{\Psi}, \underline{\mathbf 1},\emptyset ,(0), (\lie h(\mathring{\Psi})^{\perp})\oplus \mathbb{C} d$ with $\mathring{\Psi}\subsetneq \mathring{\Delta}$ being a maximal closed subroot system is a maximal symmetric regular subalgebra. To see this, let \begin{equation}\label{2rrfg}\lie s(\mathring{\Psi}, \underline{\mathbf 1},\emptyset ,(0), (\lie h(\mathring{\Psi})^{\perp})\oplus \mathbb{C} d\subseteq \lie s'':=\lie s(\mathring{\Psi}'', \underline{\mathbf k''},A'' ,(f_i''), (V_x''))\oplus \mathbb{C} d.\end{equation} Hence $\mathring{\Psi}\subseteq \mathring{\Psi}''$ which implies $\mathring{\Psi}''=\mathring{\Delta}$ or $\mathring{\Psi}=\mathring{\Psi}''$. In the first case we must have $s''=1$ and $k_1''=1$ (recall that $\lie s''$ contains already all the imaginary root spaces). So with Remark~\ref{symmremark} and \eqref{rootgenaff} we get $$\lie s''=\lie g(\Delta(\lie s'')^{\mathrm{re}})\oplus \mathbb{C} d=[\lie g,\lie g]\oplus \mathbb{C} d=\lie g.$$
In the latter case $\mathring{\Psi}=\mathring{\Psi}''$ we obviously have equality in \eqref{2rrfg}.

\textit{Case 2}: Suppose that $\mathring{\Psi}=\mathring{\Delta}$. Again by Remark~\ref{symmremark}, we get $\lie s=\lie g(\Delta(\lie s)^{\mathrm{re}})\oplus \mathbb{C} d=
\lie s(\mathring{\Delta}, k_1, \emptyset, f, (0))\oplus \mathbb{C} d$.
If $k_1=1$ then we have $\lie g(\Delta(\lie s)^{\mathrm{re}})=[\lie g, \lie g]$ using \eqref{rootgenaff} which is a contradiction. So we must have $k_1\neq 1$ and we choose a prime divisor $k$ of $k_1$. Then we have $\lie s\subseteq \lie s(\mathring{\Delta}, k, \emptyset, f, (0))\oplus \mathbb{C} d\subsetneq \lie g$ and by the maximality we can conclude $\lie s = \lie s(\mathring{\Delta}, k, \emptyset, f, (0))\oplus \mathbb{C} d$. 

\medskip
Conversely, any $\lie s(\mathring{\Delta}, k, \emptyset, f, (0))\oplus \mathbb{C} d$ (with $k$ prime) must be maximal. To see this, assume that 
$\lie s(\mathring{\Delta}, k, \emptyset, f, (0))\oplus \mathbb{C} d\subseteq {\lie s}''$.  
Since $\Delta(\lie s'')\cap \mathring{\Delta}=\mathring{\Delta}$, we must have 
$\lie s''=\lie s(\mathring{\Delta}, k', \emptyset, f', (0)) \oplus \mathbb{C} d$. Again if $k'=1$, then we get $\lie s''=\lie g$. So assume $k'\neq 1.$ In this case, we
 immediately get
$$\text{
$f(\alpha)+k\mathbb{Z}\subseteq f'(\alpha)+k'\mathbb{Z}$ for all $\alpha\in \mathring{\Delta}$}.$$
This implies that $k=k'$ since $k$ is prime and $f(\alpha)+k\mathbb{Z} = f'(\alpha)+k\mathbb{Z}$ for all $\alpha\in \mathring{\Delta}$. Thus 
 $\lie s(\mathring{\Delta}, k, \emptyset, f, (0))\oplus \mathbb{C} d = {\lie s}''$ and the proof is completed.
\end{proof}

        %%%%%%%%%%%%%%%%%%%%%%%%%%%%%%%%%%%%%%%%%%%%%%%%%%%%%%%%%%%%%%%%%%%%%%%%%%%%%%%%%%%%%%%%%%%%%%%%%%%%%%%%%%%%%%%%%%%%%%%%%%%%%%%%%%%%%%%%%%%%%%%%%%%%%%%%%%%%%%%%%%%%%%%%%%%%%%%%%%%%
        %%%%%%%%%%%%%%%%%%%%%%%%%%%%%%%%%%%%%%%%%%%%%%%%%%%%%%%%%%%%%%%%%%%%%%%%%%%%%%%%%%%%%%%%%%%%%%%%%%%%%%%%%%%%%%%%%%%%%%%%%%%%%%%%%%%%%%%%%%%%%%%%%%%%%%%%%%%%%%%%%%%%%%%%%%%%%%%%%%%%
        %%%%%%%%%%%%%%%%%%%%%%%%%%%%%%%%%%%%%%%%%%%%%%%%%%%%%%%%%%%%%%%%%%%%%%%%%%%%%%%%%%%%%%%%%%%%%%%%%%%%%%%%%%%%%%%%%%%%%%%%%%%%%%%%%%%%%%%%%%%%%%%%%%%%%%%%%%%%%%%%%%%%%%%%%%%%%%%%%%%%
        %%%%%%%%%%%%%%%%%%%%%%%%%%%%%%%%%%%%%%%%%%%%%%%%%%%%%%%%%%%%%%%%%%%%%%%%%%%%%%%%%%%%%%%%%%%%%%%%%%%%%%%%%%%%%%%%%%%%%%%%%%%%%%%%%%%%%%%%%%%%%%%%%%%%%%%%%%%%%%%%%%%%%%%%%%%%%%%
\begin{cor}\label{maincorunt12}
We have a one-to-one correspondence between maximal closed subroot systems and maximal symmetric regular subalgebras different from $[\lie g,\lie g]$.
\qed
\end{cor}

\section{Remarks on the nonsymmetric case}\label{section6}
If $\lie s$ is not necessarily symmetric, we define 
\begin{equation*}\label{dec224}\Delta(\lie{s})_{\mathrm{sy}}:=\{\alpha\in \Delta(\lie{s}):-\alpha\in \Delta(\lie{s})\},\ \ \Delta(\lie{s})_{\mathrm{sp}}:=\Delta(\lie{s})\backslash \Delta(\lie{s})_{\mathrm{sy}}\end{equation*}
and refer to the above subsets as the \textit{symmetric} and \textit{special} parts respectively. However, the description of the special part is even unknown in the case $|\Delta|<\infty$ and some progress has been made in \cite{DdeG21} giving algorithms how to produce special roots. In this section we recall Dynkin's result, see that it fails for affine Kac-Moody algebras and prove an anlogue under further restrictions.
\subsection{} We set
$$\lie{s}_{\mathrm{sy}}:=\langle \lie{s}_\alpha: \alpha\in \Delta(\lie{s})_{\mathrm{sy}} \rangle, \ \ \  \lie{s}_{\mathrm{sp}}=\bigoplus_{\alpha\in\Delta(\lie s)_{\mathrm{sp}}} \lie s_{\alpha}\oplus  \lie h_{\lie s_\mathrm{sp}}$$ 
where $\lie h_{\lie{s}_\mathrm{sp}}$ is a vector space complement of $\lie h_{\lie{s}_\mathrm{sy}}:=(\lie h\cap \lie s_{\mathrm{sy}})$ in  $\lie h_{\lie s}:=(\lie h\cap \lie s)$. Dynkin studied regular subalgebras of finite dimensional semi-simple Lie algebras in \cite{dynkin1952semisimple} and obtained the following result.
\begin{thm}\label{Dynkintheorem}
Let $\lie{g}$ be a finite-dimensional semi-simple Lie algebra with Cartan subalgebra $\lie{h}$. Given a regular subalgebra $\lie{s}$ of $\lie{g}$, we have that 
$\lie{s}_{\mathrm{sp}}$ is the radical of $\lie{s}$ and $\lie{s}_{\mathrm{sy}}$ is a maximal semi-simple Lie algebra of $\lie{s}$ such that 
$\lie{s}=\lie{s}_{\mathrm{sp}}\rtimes \lie{s}_{\mathrm{sy}}.$
\qed
\end{thm}

The analogue for Kac-Moody algebras is false in general as the following example illustrates.
\begin{example} Let $\alpha$ be a root of $\lie{sl_4}(\mathbb{C})$ and $h'$ be a diagonal matrix in $\lie{sl_4}(\mathbb{C})$ such that $\alpha(h')=0$. %For $r\in \mathbb{N}$ define

Define the regular subalgebra $\lie s$ of $\lie g=A_3^{(1)}$ as follows:
$$\lie s:=\left(\bigoplus_{r\in 2\bn} \bc(x_\alpha\otimes t^{r})\oplus \bigoplus_{r\in 2\bn} \bc(x_{-\alpha}\otimes t^{r})\oplus \bigoplus_{r\in 2\bn} \bc(\alpha^\vee\otimes t^{r})\right)\bigoplus \bc (h'\otimes t^{-2}).$$
We have 
$$\Delta(\lie s)_{\mathrm{sp}}=\{\alpha+r\delta,-\alpha+r\delta:r\in 2\bn\}\cup\{r\delta:r\in 2\bn\backslash\{2\}\},\ \ \Delta(\lie s)_{\mathrm{sy}}=\{\pm 2\delta\}.$$
However, $\lie s_{\mathrm{sp}}$ is not even a subalgebra of $\lie g$ since $[x_\alpha\otimes t,x_{-\alpha}\otimes t]=\alpha^\vee\otimes t^{2}\notin {\lie s}_{\mathrm{sp}}$. We emphasize here that the restriction of the bilinear form $(\cdot,\cdot): \lie s_{2\delta}\times\lie s_{-2\delta}\rightarrow \mathbb{C}$ is degenerate.

\end{example}
The analogue of Dynkin result reads as follows. 
\begin{prop}\label{dynanalogue}  Let $\lie{s}$ be a regular subalgebra of $\lie g$ such that the restriction of the bilinear form
    $$(\cdot,\cdot): \lie s_{\alpha}\times \lie s_{-\alpha} \rightarrow \mathbb{C},\ \ \alpha\in \Delta(\lie{s})_{\mathrm{sy}}$$
    remains non-degenerate and the Chevalley involution $\omega$ satisfies $\omega(\lie s_{\alpha})\subseteq \lie s_{-\alpha}$ for all $\alpha\in\Delta(\lie s)_{\mathrm{sy}}$.
     We have that $\lie{s}_{\mathrm{sp}}$  is an ideal of $\lie{s}$ and 
 \begin{equation}\label{inflev}\lie{s}_{\mathrm{sy}}=\bigoplus_{\alpha\in\Delta(\lie s)_{\mathrm{sy}}} \lie s_{\alpha} \oplus \lie h_{\lie s_\mathrm{sy}},\  \
 \ \lie{s}=\lie{s}_{\mathrm{sp}}\rtimes \lie{s}_{\mathrm{sy}}.\end{equation}
    \end{prop}
    \begin{proof} Let $\alpha,\beta\in \Delta(\lie s)_{\mathrm{sy}}$ such that $[\lie{s}_{\alpha},\lie{s}_{\beta}]\neq 0.$ If $\alpha=-\beta$ we have $[\lie{s}_{\alpha},\lie{s}_{\beta}]\in \lie{h}_{\lie s_\mathrm{sy}}$. Otherwise, by applying the Chevalley involution we also get that $-(\alpha+\beta)\in \Delta(\lie s)$ and thus $\alpha+\beta\in \Delta(\lie s)_{\mathrm{sy}}$. This shows the first equation in \eqref{inflev}. Next we show that $\lie{s}_{\mathrm{sp}}$ is an ideal of $\lie{s}.$ Let $\alpha\in \Delta(\lie s)_{\mathrm{sp}}$ and $\beta\in \Delta(\lie s)$ such that $[\lie s_{\alpha},\lie s_{\beta}]\neq 0$; note that in this case $\alpha\neq-\beta$. So we can choose $u_{\alpha}\in \lie s_{\alpha}$ and $u_\beta\in \lie s_{\beta}$ such that $u_{\alpha+\beta}:=[u_{\alpha},u_{\beta}]\in \lie s_{\alpha+\beta}$ is non-zero. If $\alpha+\beta\in \Delta(\lie s)_{\mathrm{sp}}$ we are done. Otherwise $\alpha+\beta\in \Delta(\lie s)_{\mathrm{sy}}$ and by the non-degeneracy of the form we can find a non-zero element $y_{-\alpha-\beta}\in \lie s_{-\alpha-\beta}$ such that $(u_{\alpha+\beta},y_{-\alpha-\beta})\neq 0$. This gives in particular $[\lie s_{-\alpha-\beta},\lie s_{\beta}]\neq 0$ since otherwise 
        $$0\neq (u_{\alpha+\beta},y_{-\alpha-\beta})=([u_{\alpha},u_{\beta}],y_{-\alpha-\beta})=(u_{\alpha},[u_{\beta},y_{-\alpha-\beta})])=0$$
        which is a contradiction. Hence $[\lie{s}_{-\alpha-\beta},\lie{s}_{\beta}]\neq 0$ giving $-\alpha\in \Delta(\lie s)$ which is once more a contradiction. So we must have $\alpha+\beta\in \Delta(\lie s)_{\mathrm{sp}}$.
    \end{proof}
Note that both assumptions made in the above proposition are trivially satisfied for all real roots in $\Delta(\lie s)$; in particular finite root system.
\subsection{}  In this subsection we will discuss some open questions.
\begin{itemize}
\item Let $\Psi$ be a maximal closed subroot system. Is it true that the poset (poset structure given by inclusion)
$$A_{\Psi}:=\{\lie s: \Delta(\lie s)^{\mathrm{re}}=\Psi\}$$ has a unique maximal element? This is true in the untwisted affine case. Note that $A_{\Psi}$ satisfies the following condition:
$$\lie s \subseteq \lie s'\implies \lie s'\in A_{\Psi}$$
for all $\lie s\in A_{\Psi}$ and symmetric regular subalgebras $\lie s'$.
\item Is there a one-to-one correspondence between maximal real closed subroot systems and maximal symmetric regular subalgebras (different from the derived algebra) as in the untwisted affine case?
Our calculations suggest that the maximal candidates satisfy
  $$\lie s=\lie h\oplus \sum_{\alpha\in \Delta(\lie s)} \lie g_{\alpha}.$$
\item Classify real closed subroot systems $\Psi$ such that $\Pi(\Psi)$ is linearly independent (c.f. Remark~\ref{remlininde1}). Computations suggest that for rank 2 hyperbolic Kac-Moody algebras every $\pi$-system $\Sigma\subseteq \Delta^+$ satisfies $|\Sigma|\le 2.$ In particular, $\Sigma$ is linearly independent (c.f. Example~\ref{exrank2}).
\end{itemize}

\bibliographystyle{plain}
\bibliography{bibliography}

\end{document}